\documentclass[11pt,reqno]{amsart}

\usepackage[T1]{fontenc}
\usepackage[utf8]{inputenc}
\usepackage{lmodern}
\usepackage{microtype}
\usepackage{mathtools,amssymb,mathrsfs}
\usepackage{enumitem,booktabs,array}
\usepackage{xcolor,tikz-cd}
\usepackage{hyperref}
\usepackage[capitalise,noabbrev]{cleveref}

\hypersetup{
  colorlinks=true,
  linkcolor=blue!55!black,
  citecolor=green!45!black,
  urlcolor=blue!60!black,
  pdftitle={Torus-enriched motivic Bruhat complexes and maximal compact groups},
  pdfauthor={Tianle Liu}
}
\setlist[itemize]{leftmargin=2em}
\setlist[enumerate]{leftmargin=2.2em}
\allowdisplaybreaks

\usepackage{aliascnt}
\newtheorem{theorem}{Theorem}[section]
\newcommand{\newaliasthm}[2]{%
  \newaliascnt{#1}{theorem}%
  \newtheorem{#1}[#1]{#2}%
  \aliascntresetthe{#1}%
  \expandafter\def\csname cref@#1@alias\endcsname{#1}%
  \crefname{#1}{#2}{#2s}%
}
\newaliasthm{proposition}{Proposition}
\newaliasthm{lemma}{Lemma}
\newaliasthm{corollary}{Corollary}
\theoremstyle{definition}
\newaliasthm{definition}{Definition}
\newaliasthm{construction}{Construction}
\newaliasthm{example}{Example}
\theoremstyle{remark}
\newaliasthm{remark}{Remark}

\DeclareMathOperator{\ReR}{Re_{\mathbb R}}
\DeclareMathOperator{\sgn}{sgn}

\DeclareMathOperator{\Th}{Th}
\DeclareMathOperator{\gr}{gr}
\newcommand{\Aone}{\mathbb A^1}
\newcommand{\Ga}{\mathbb G_a}
\newcommand{\Gm}{\mathbb G_m}
\newcommand{\Z}{\mathbb Z}
\newcommand{\R}{\mathbb R}
\newcommand{\KMW}[1]{\mathbf K_{#1}^{\mathrm{MW}}}

\newcommand{\ZA}{\mathbf Z_{\Aone}}
\newcommand{\HT}{\mathbf H_T}
\newcommand{\CellC}{\mathbf C^{\mathrm{cell}}}
\newcommand{\CellH}{\mathbf H^{\mathrm{cell}}}
\newcommand{\GW}{\mathbf{GW}}
\newcommand{\eps}{\varepsilon}
\newcommand{\epsPS}{\eps^{\mathrm{PS}}}
\newcommand{\ol}[1]{\overline{#1}}
\newcommand{\ang}[1]{\langle #1\rangle}
\newcommand{\Dop}{\mathscr D}

\makeatletter
\@ifundefined{subjclassname@2020}{%
  \@namedef{subjclassname@2020}{%
    \textup{2020} Mathematics Subject Classification}%
}{}
\makeatother

\title[Torus-enriched motivic Bruhat complexes]
  {Torus-enriched motivic Bruhat complexes\\
   and maximal compact groups}
\author{Haoyang Liu and Tianle Liu}
\date{}
\address{University of California, Santa Barbara}
\email{haoyangliu@ucsb.edu}
\address{University of Southern California}
\email{tianleli@usc.edu}

\begin{document}

\begin{abstract}
Bruhat decompositions give cellular models for split algebraic groups,
flag varieties, and maximal compact groups, but motivic boundaries
retain orientation and torus-translation data lost in the flag
quotient.  Over a perfect field of characteristic zero, let the group
be connected, split, semisimple, and simply connected.  Fixing a Borel
subgroup with split maximal torus and unipotent radical, we construct a
torus-enriched motivic cellular complex for the basic affine space and
compute its boundary in every degree.  Each cover in Bruhat order
contributes a two-face operator determined by a transported coroot, a
tail determinant weight, and an explicit Milnor--Witt frame degree.
Bott--Samelson purity proves the formula, while the unipotent torsor
identifies the complex with that of the group.  Over the real numbers,
realization identifies it at chain level with the extended-Weyl complex
of a maximal compact subgroup, while torus augmentation gives the flag
complex.  A single motivic complex therefore interpolates between the
two incidence theories.  A finite torus-support filtration makes this
explicit; after inversion of two it splits by the characters of the
component group of the real split torus, and the support spectral
sequence degenerates.  Calculations in the rank-three special linear
and exceptional rank-two cases exhibit the first higher differentials
beyond the previously known range.
\end{abstract}

\maketitle
\tableofcontents

\section{Introduction}

Bruhat decompositions provide a common combinatorial skeleton for split
algebraic groups, flag varieties, and maximal compact groups.  Their
cellular incidence maps nevertheless appear in different forms:
motivic boundaries carry quadratic orientation data, compact-group
boundaries retain discrete torus labels, and passage to the flag
quotient combines those labels into the classical coefficients
\(0,\pm2\).  The purpose of this paper is to construct a single
torus-enriched motivic complex that accounts for all three phenomena
and whose boundary is explicit in every cellular degree.

Let \(k\) be a perfect field of characteristic zero and let \(G\) be a
connected, split, semisimple, and simply connected linear algebraic
\(k\)-group.  Fix a Borel \(k\)-subgroup \(B\subset G\) and a split
maximal \(k\)-torus \(T\subset B\), write \(U=R_u(B)\), so that
\(B=T\ltimes U\), and choose a split pinning
\((T,B,\{x_\alpha:\mathbb G_a\xrightarrow{\sim}U_\alpha\}_{\alpha\in\Delta})\).
Thus \(\Delta\) is the set of simple roots determined by \(B\), and the
Weyl group is \(W=N_G(T)/T\).  Our basic geometric object is the
cellular \(\Aone\)-chain complex of the basic affine space
\[
  X=G/U,
\]
which also models the group \(G\) through the \(U\)-torsor
\(G\to G/U\).  The Bruhat strata of \(X\) have the form
\[
  BwU/U\cong\mathbb A^{\ell(w)}\times T,
\]
so the motivic cell over \(w\) carries the Hopf algebra
\[
  \HT:=\ZA[T]
\]
of the split torus.  This coefficient algebra separates the two faces
of a cell deletion.  After real realization it becomes the group ring
\(\Z[\pi_0(T(\R))]\) that labels the cells of a maximal compact
subgroup; torus augmentation instead combines the two faces and
produces the flag boundary.

Morel and Sawant introduced cellular \(\Aone\)-homology and constructed
the Bruhat cellular complex of a split reductive group in
\cite{MorelSawant2023}.  They computed its rank-one differential and its
differential in degrees one and two, using the answer to prove a motivic
version of Matsumoto's theorem.  The present paper extends their
explicit low-degree calculation to a closed formula in every degree.

The all-degree formula has two specializations, and it was designed
with both in mind.  We describe the real one first.  Let \(k=\R\), put
\(A=T(\R)^0\), and let \(K\subset G(\R)\) be the pinning-compatible
maximal compact subgroup.  (For split simply connected semisimple
\(G\), the group \(G(\R)\) is automatically connected; see
\Cref{rem:GR-connected}.)  Patr\~ao and Sandoval constructed a Bruhat
CW structure on \(K\) in \cite{PatraoSandoval2026}.
Its cells are indexed by an extension of the Weyl group by
\[
  M=Z_K(A)\cong\pi_0(T(\R)),
\]
and its boundary is more refined than the familiar boundary of the real
flag manifold \(K/M\).  A deletion in a reduced word has two faces.
Before passing to \(K/M\), these faces usually land in two different
\(M\)-labelled cells and have individual coefficients \(\pm1\); only
after forgetting the label do they combine to \(0\) or \(\pm2\).

The second specialization is the flag variety.  For split flag
varieties, the companion preprint \cite{LiuLiu2026} gives, \emph{under
its boundary-data hypothesis} \cite[Definition~2.1]{LiuLiu2026} and in
the range of the signed boundary data constructed there,
\[
  \partial_i^{\Aone}=\eta E_i\qquad(i\geq2),
\]
where \(E_i\) is a length-reindexed half of the integral boundary
matrix \(\partial^{\mathrm{top}}\) of the real flag CW complex.  (Here
and below, \(\partial^{\mathrm{top}}\) denotes the integral cellular
boundary matrix of the real points with their Bruhat CW structure.)
This range restriction is inherited by every statement of the present
paper that refers to the matrices \(E_i\); the torus-enriched formula
itself is proved in all degrees and does not depend on it.

The prescription \(\eta\,\partial^{\mathrm{top}}/2\) cannot be applied
to \(K\).  Already for \(SL_2(\R)\), the maximal compact is a circle
divided into two vertices and two edges, and the group boundary has
entries \(1\) and \(-1\).  There is no integral half-boundary.  The
missing structure is the split torus: Iwasawa decomposition gives
\(X(\R)\cong K\times A\) with \(A\) contractible, and the real
realization of \(\HT\) is \(\Z[M]\), the coefficient ring needed to
remember the two faces separately.

\newpage
\subsection*{Main results and unifying structure}

The central point is not merely that motivic, compact-group, and
flag-variety calculations give compatible numerical answers.  They are
connecting morphisms of one torus-enriched filtered construction.
Over \(\R\), the resulting chain complexes are organized by the
following square:
\begin{equation}\label{eq:intro-unifying-square}
\begin{tikzcd}[column sep=large,row sep=large]
 \CellC_*(G)\cong\CellC_*(X)
   \arrow[r,"-\otimes_{\HT}\Z"]
   \arrow[d,"\operatorname{Re}_{\R}^{\mathrm{fil}}"']
 &
 \CellC_*(G/B)
   \arrow[d,"\operatorname{Re}_{\R}^{\mathrm{fil}}"]
 \\
 C_*^{\mathrm{CW}}(K;\Z)
   \arrow[r,"-\otimes_{\Z[M]}\Z"']
 &
 C_*^{\mathrm{CW}}(K/M;\Z).
\end{tikzcd}
\end{equation}
The vertical arrows in \eqref{eq:intro-unifying-square} mean realization
of the filtered motivic spaces and their purity cofibers, followed by
singular chains and passage to the exact couple; they are not termwise
realization functors on complexes of strictly \(\Aone\)-invariant
sheaves.  The upper horizontal arrow is torus augmentation, while the
lower one identifies the two \(M\)-labelled faces.  Thus the same
geometric connecting morphism retains both faces on the left and
combines them on the right.

The paper has three main structural results, all reflected in
\eqref{eq:intro-unifying-square}.  First,
\Cref{thm:G-X} identifies the Morel--Sawant group complex with the
basic-affine complex, and \Cref{thm:all-degree} computes its boundary in
all degrees by combining a global two-stratum purity model with a local
rank-one clutching calculation.  This yields
\eqref{eq:intro-boundary}, including the torus translation, determinant
weight, and frame degree attached to every Bruhat cover.

Second, \Cref{thm:real-comparison} identifies the real realization of
this complex with the extended-Weyl complex of a maximal compact
subgroup, whereas \Cref{thm:flag-augmentation} identifies torus
augmentation with the flag complex.  The comparison square of
\Cref{thm:comparison-square} places the two incidence theories under
one chain-level construction.  It also produces the compact
fundamental cycle of \Cref{cor:compact-fundamental-cycle}; within the
boundary-data range of \cite{LiuLiu2026}, \Cref{cor:eta} recovers the
flag \(\eta\)-boundary.

Third, the finite torus-support filtration of
\Cref{prop:support-ss,cor:real-support-graded} has associated graded
given explicitly by parity-twisted flag complexes.  After real
realization and inversion of \(2\), \Cref{thm:half-splitting} splits
the filtration along the characters of \(M\) and proves degeneration
of the support spectral sequence at \(E^1\).

The examples test each part of the structure beyond the previously
known range.  The complete \(SL_4\) degree-three differential in
\Cref{prop:SL4-degree-three} is the first new differential after the
degrees computed by Morel--Sawant.  The rank-two frame calculation
\Cref{prop:rank-two-degrees} includes the six-term \(G_2\) braid, and
\Cref{sec:G2} gives the complete exceptional complex.  The resulting
integer homology agrees with the corresponding maximal compact and
flag manifolds.

We work over a perfect field of characteristic zero in order to use a
uniform pinned-coordinate and Milnor--Witt orientation framework.  The
Chevalley constants \(2\) and \(3\) occur in strict triangular terms,
but treating bad characteristic would require separate geometric and
orientation arguments and is not pursued here.

\subsection*{The local operator}

We now give the explicit cover operator underlying the preceding
structure.  Let \(a=s_{i_1}\cdots s_{i_p}\) be a reduced word and suppose
deleting position \(q\) leaves a reduced word for \(b\).  Write
\[
  t_q=s_{i_{q+1}}\cdots s_{i_p},\qquad
  \lambda_{a,q}=t_q^{-1}(\alpha_{i_q}^{\vee}),
\]
and
\[
  \sigma_{a,q}=
  \sum_{\beta\in\Phi^+\cap t_q\Phi^-}
  \langle\beta,\alpha_{i_q}^{\vee}\rangle.
\]
The coroot \(\lambda_{a,q}\) records the second torus-labelled face.  The
integer \(\sigma_{a,q}\) is the determinant weight obtained when the
deleted rank-one torus is moved through the tail root groups.  We define
\(\Dop_{\lambda,\sigma}\) by
\begin{equation}\label{eq:intro-operator}
  \Dop_{\lambda,\sigma}
  \big((u)\otimes x\otimes[t]\big)
  =\ang{u^\sigma}x\otimes[\lambda(u)t]-x\otimes[t].
\end{equation}
This is the twisted reduced class obtained by combining torus
translation with the determinant character
\(u\mapsto\ang{u^\sigma}\).  Only for \(\sigma=0\) is it ordinary
multiplication by \([\lambda(u)]-[1]\).  Its geometric construction on
Thom spaces, given in \Cref{constr:D}, defines a morphism of strictly
\(\Aone\)-invariant sheaves.

If \(\vartheta_{a,q}\in\GW(k)^\times\) is the Milnor--Witt degree of the
change from the deleted word to the chosen normal word for \(b\), the
cover component is
\begin{equation}\label{eq:intro-boundary}
  \partial_{a,b}=(-1)^{q-1}\vartheta_{a,q}
  \Dop_{\lambda_{a,q},\sigma_{a,q}}.
\end{equation}
Here the symmetry that separates the deleted Thom factor is understood;
its precise convention is fixed in \Cref{constr:cut}.
Formula \eqref{eq:intro-boundary} is valid in every cellular degree.  In
degrees one and two it expands to the formulas of Morel--Sawant.
If several deletion positions \(q\) produce the same target \(b\), the
matrix entry \(\partial_{a,b}\) is the sum of the displayed
contributions over those positions.

There are two natural simplifications.  Over \(\R\), put
\(c_{a,q}=\lambda_{a,q}(-1)\in M\).  Formula
\eqref{eq:intro-operator} becomes right multiplication by
\[
  (-1)^{\sigma_{a,q}}c_{a,q}-1\in\Z[M].
\]
After taking the signature of \(\vartheta_{a,q}\), this is precisely the
Patr\~ao--Sandoval coefficient.  On the other hand, after the torus
augmentation \(\HT\to\Z\), it becomes
\[
  (u)\otimes x\longmapsto(\ang{u^\sigma}-1)x,
\]
the Milnor--Witt power-map boundary.  The flag \(\eta\)-formula is
therefore an augmentation of the more primitive group formula, not a
recipe from which the latter can be recovered.  These are the two
specializations already organized by
\eqref{eq:intro-unifying-square}.

\subsection*{Relation to previous work}

The existence of the complex for \(G\), its bimodule structure, and its
low-degree calculation are due to Morel--Sawant
\cite{MorelSawant2023}.  Patr\~ao--Sandoval \cite{PatraoSandoval2026}
constructed the extended-Weyl CW structure on the classical compact
group \(K\) and computed its full integral incidence formula.  The
present comparison identifies that topological complex with the real
realization of the torus-enriched motivic construction.  The companion
preprint \cite{LiuLiu2026} computes flag-variety
\(\Aone\)-differentials within its boundary-data range.  Its use here
is confined to \Cref{prop:flag-orientation-dictionary,cor:eta}, which
retain that hypothesis, and to numerical comparisons; the all-degree
boundary formula, compact comparison, and rank-two frame degrees are
proved independently.  The basic-affine model supplies the common
geometric object: it retains the torus labels needed for the compact
complex and admits augmentation to the flag complex.

Liu--Peng compute cellular \(\Aone\)-homology for smooth toric
varieties and, more specifically, the action of toric group sections
on oriented cubical cells
  \cite[Lemma~2.22, Proposition~2.23, and Proposition~3.1]{LiuPeng2025}.
  Their formulas
provide the toric analogue of the determinant, character-translation,
and lower-face terms in our universal operator.  We use their
rank-one transition calculation as an independent check in
\Cref{rem:Liu-Peng}; the Bruhat two-stratum localization and the
filtered Iwasawa comparison proved here are not consequences of the
toric calculation.

On the real side, the parity phenomenon itself is classical.  The
\(0,\pm2\) boundary coefficients of real flag manifolds and their
control by root heights go back to Kocherlakota's Morse-theoretic
computation of the integral homology of real flag manifolds
\cite{Kocherlakota1995}; cellular proofs and refinements were given by
Rabelo--San Martin \cite{RabeloSanMartin2019}, the root-height
correspondence was made explicit by Lambert--Rabelo
\cite{LambertRabelo2022}, and the cohomological structure, including
the absence of higher \(2\)-power torsion in the even-flag range, was
studied by Matszangosz \cite{Matszangosz2021}; the corresponding
type-\(A\) statement and its motivic explanation were established by
Hudson--Matszangosz--Wendt \cite{HudsonMatszangoszWendt2024}.  This
classical parity rule is the flag-level shadow of the lift constructed
here: the motivic formula retains the two torus-labelled faces whose
difference produces the rule, works over an arbitrary
characteristic-zero base, and interpolates between the compact and flag
complexes.  Rational motives of reductive groups and
several Milnor--Witt motivic cohomology calculations are known by other
methods \cite{Biglari2012,Peng2024}; they do not determine this integral
torus-enriched boundary.

\section{Cellular chains and torus coefficients}
\label{sec:prelim}

Throughout, \(k\) is perfect of characteristic zero and \(G/k\) is split
semisimple and simply connected.  We fix a pinning
\[
  \bigl(T,B,
  \{x_\alpha:\Ga\xrightarrow{\sim}U_\alpha\}_{\alpha\in\Delta}\bigr),
  \qquad B=TU.
\]
The pinning also fixes the opposite root parameters and the rank-one
homomorphisms \(SL_2\to G\).  We order the simple roots.  Write
\(W=N_G(T)/T\), let
\(w_0\) be the longest element, and put \(N=\ell(w_0)=|\Phi^+|\).

We work in the abelian category of strictly \(\Aone\)-invariant
Nisnevich sheaves.  The free strictly \(\Aone\)-invariant Nisnevich
sheaf on a smooth scheme \(Y\) is denoted by
\(\ZA[Y]\).  For a pointed scheme \((Y,y)\), parentheses denote the
reduced sheaf
\(\ZA(Y)=\ker(\ZA[Y]\to\Z)\), where the displayed map is the
augmentation and \(y\) supplies its splitting; in particular
\(\ZA(\Gm)=\KMW1\), with base point \(1\).  Our Milnor--Witt conventions
are
\[
  (uv)=(u)+(v)+\eta(u)(v),\qquad
  \ang u=1+\eta(u),\qquad
  \ang{-1}\eta=-\eta.
\]
See \cite{Morel2012,MorelSawant2023}.  For \(n\in\Z\), the
\emph{signed power degree} is
\[
 n_\eps=\sum_{j=1}^{n}\ang{(-1)^{j-1}}\ \ (n\geq0),\qquad
 n_\eps=-\ang{-1}(-n)_\eps\ \ (n<0);
\]
thus \(n_\eps\in\GW(k)\) is the Milnor--Witt degree of the \(n\)-th
power map on \(\Gm\); see \cite[Section~3.1]{Morel2012}.  Concretely,
\(n_\eps=\tfrac n2h\) for even \(n\) and
\(n_\eps=\tfrac{n-1}2h+1\) for odd \(n\), where
\(h=1+\ang{-1}\); in particular \(n_\eps\eta=0\) for \(n\) even and
\(n_\eps\eta=\eta\) for \(n\) odd.

All tensor products of strictly \(\Aone\)-invariant sheaves below are
the \(\Aone\)-localized tensor products in this abelian category.  We
suppress the localization symbol.

\begin{lemma}\label{lem:MW-kunneth}
For \(r\geq1\), consider the canonical epimorphisms
\[
 \ZA(\Gm)\longrightarrow\KMW1,
 \qquad
 \ZA(\Gm^{\wedge r})\longrightarrow\KMW r.
\]
Their \(\Aone\)-localized tensor product induces an isomorphism
\begin{equation}\label{eq:MW-tensor}
 \KMW1\otimes\KMW r\xrightarrow{\ \sim\ }\KMW{r+1},
\end{equation}
which on sections over a field \(F\) is symbol concatenation
\((u)\otimes(v_1)\cdots(v_r)\mapsto(u)(v_1)\cdots(v_r)\).  Under the
symmetry isomorphism of the localized tensor product, interchanging the
two factors of \(\KMW1\otimes\KMW1\) corresponds under
\eqref{eq:MW-tensor} to multiplication by
\(\eps=-\ang{-1}\) on \(\KMW2\).
\end{lemma}

\begin{proof}
By \cite[Theorem~3.37]{Morel2012}, the free strictly
\(\Aone\)-invariant sheaf on the pointed smash power
\(\Gm^{\wedge r}\) is \(\KMW r\); equivalently, \(\KMW r\) is the
\(r\)-fold contracted free sheaf and represents
\(\ZA(\Gm)^{\otimes r}\) after \(\Aone\)-localization, cf.\
\cite[Section~2.2]{MorelSawant2023}.  The universal
property of the localized tensor product then gives
\eqref{eq:MW-tensor}, and evaluation on field points is symbol
concatenation because both epimorphisms send a unit \(u\) to its symbol
\((u)\).  The \(\eps\)-commutativity of symbol multiplication,
\((u)(v)=\eps(v)(u)\), is proved in
\cite[Section~3.1]{Morel2012}; since the
symmetry of the tensor product is computed on the representing smash
factors, it acts by \(\eps\) as asserted.
\end{proof}

The inverses of \eqref{eq:MW-tensor}, with the symmetry isomorphisms in
the indicated factor order, are what is used below to cut a Thom symbol
into ordered \(\KMW1\)-factors; the precise convention is fixed in
\Cref{constr:cut}.

We also record the Thom-determinant fact used throughout.

\begin{lemma}
\label{lem:thom-determinant}
Let \(Y\) be a smooth scheme, \(V\) a trivialized rank-\(c\) vector
bundle on \(Y\), and \(g:V\to V\) a bundle automorphism given in the
trivialization by a morphism \(g:Y\to GL_c\).  Then the induced pointed
endomorphism of \(\Th(V)=V/(V-0)\) acts on
\(\widetilde{\mathbf H}^{\Aone}_*\bigl(\Th(V)\bigr)
\cong\KMW c\otimes\ZA[Y]\)
by
\[
 x\otimes[y]\longmapsto\ang{\det g(y)}\,x\otimes[y].
\]
In particular the action depends only on the determinant composite
\(\det g:Y\to\Gm\).
\end{lemma}

\begin{proof}
Both the endomorphism induced by \(g\) and the endomorphism displayed
in the statement are morphisms of strictly \(\Aone\)-invariant
sheaves.  Such morphisms are determined by their values on finitely
generated separable fields: strictly \(\Aone\)-invariant sheaves are
unramified, and their sections on a smooth integral scheme inject into
the sections over its function field; see
\cite[Chapter~2]{Morel2012}.  It is therefore enough to evaluate after
a field extension \(F/k\) and on the generators
\(x\otimes[y]\), \(y\in Y(F)\), supplied by the universal property of
the free strictly \(\Aone\)-invariant sheaf \(\ZA[Y]\).

Now \(g(y)\in GL_c(F)\).  Gaussian elimination over the field \(F\)
writes it as a product of elementary matrices and
\(\operatorname{diag}(\det g(y),1,\ldots,1)\).  An elementary matrix is
\(\Aone\)-homotopic to the identity through
\(E_{ij}(sa)\), \(s\in\mathbb A^1\), so it acts trivially on the Thom
homology of the fiber.  A diagonal unit \(d\) acts on the first
\(\mathbb P^1\)-factor by the degree
\(\ang d\in\GW(F)\); see \cite[Section~3.1]{Morel2012}.  Under the
ordered Thom isomorphism this is multiplication by \(\ang d\) on
\(\KMW c(F)\).  Taking \(d=\det g(y)\) proves the formula on every
field-valued point and hence proves the asserted equality of sheaf
morphisms.  Notice that this argument does not assert that an arbitrary
family \(g:Y\to GL_c\) is globally elementary or globally
\(\Aone\)-homotopic over \(Y\) to its determinant diagonal.
\end{proof}

\begin{lemma}\label{lem:purity-naturality}
Consider a Cartesian square of smooth schemes
\[
\begin{tikzcd}[column sep=large]
 Z'\arrow[r,hook]\arrow[d,"f_Z"']
 &Y'\arrow[d,"f"]\\
 Z\arrow[r,hook]&Y
\end{tikzcd}
\]
in which the horizontal arrows are regular closed immersions.  Suppose
that \(f\) is smooth; in particular the square is transverse and the
canonical map
\(N_{Z'/Y'}\to f_Z^*N_{Z/Y}\) is an isomorphism.  Then homotopy purity
gives a commutative square in the pointed motivic homotopy category
\[
\begin{tikzcd}[column sep=large]
 Y'/(Y'-Z')\arrow[r,"\simeq"]\arrow[d]
 &\Th(N_{Z'/Y'})\arrow[d]\\
 Y/(Y-Z)\arrow[r,"\simeq"]
 &\Th(N_{Z/Y}).
\end{tikzcd}
\]
For a Cartesian map of nested triples of such pairs, these squares
assemble into a map of the canonical quotient cofiber sequences and
therefore commute with their connecting morphisms.
\end{lemma}

\begin{proof}
Use the deformation-to-the-normal-cone construction of the purity
equivalence \cite[Section~3, Theorem~2.23]{MorelVoevodsky1999}.  The
blow-up defining the deformation space commutes with the displayed
transverse base change, and removing the strict transform of
\(Y\times\{0\}\) commutes with it as well.  On the generic fiber the
resulting map is \(f\), while on the special fiber it is the displayed
normal-bundle map.  The two specialization equivalences therefore give
the commutative square.  For nested triples, all maps are induced before
specialization by inclusions of the corresponding open complements.
The elementary quotient sequence
\(U_1/U_0\to U_2/U_0\to U_2/U_1\) is functorial in a nested triple;
hence specialization produces a map of cofiber sequences, including
the connecting maps.
\end{proof}

\subsection{Cellular structures}

For a cellular filtration
\[
  \varnothing=\Omega_{-1}\subset\Omega_0\subset\cdots\subset\Omega_d=Y,
\]
homotopy purity \cite[Section~3, Theorem~2.23]{MorelVoevodsky1999}
identifies each successive cofiber with the Thom space of
the normal bundle of \(S_p=\Omega_p-\Omega_{p-1}\).  After choosing an orientation,
the corresponding chain group is
\[
  \widetilde{\mathbf H}^{\Aone}_p(\Th(\nu_p))
  \cong\KMW p\otimes\ZA[S_p];
\]
compare \cite[Section~2.3]{MorelSawant2023}.
The connecting morphisms of the associated exact couple form
\(\CellC_*(Y)\).  In particular,
\(\partial^2=0\) follows from the exact couple.  A change of normal
coordinates with Jacobian determinant \(a\) multiplies the Thom generator
by \(\ang a\in\GW(k)\), by \Cref{lem:thom-determinant}.  When a chain
group is written as a direct sum over an index set, \(e_a\) denotes the
image of the canonical inclusion of the summand indexed by \(a\); thus
\(x\,e_a\) is the element \(x\) placed in the \(a\)-summand.

\subsection{The torus algebra}

Set \(\HT=\ZA[T]\).  Since \(G\) is simply connected, the simple
coroots identify \(T\cong\Gm^r\), whence
\begin{equation}\label{eq:HT-decomp}
  \HT\cong\bigotimes_{i=1}^r(\Z\oplus\KMW1)
  \cong
  \Z\oplus
  \bigoplus_{\varnothing\ne S\subseteq\{1,\ldots,r\}}\KMW{|S|}
\end{equation}
as a sheaf.  Multiplication in one torus coordinate uses
\[
  \KMW1\otimes\KMW1=\KMW2\xrightarrow{\eta}\KMW1.
\]
For a cocharacter \(\lambda:\Gm\to T\), write
\begin{equation}\label{eq:reduced-coroot}
  \ol\lambda:\KMW1=\ZA(\Gm)\longrightarrow\HT,\qquad
  (u)\longmapsto[\lambda(u)]-[1].
\end{equation}
This is the motivic analogue of \(c-1\) in a group ring.

\subsection{Normal words}

Let \(\varphi_i:SL_2\to G\) be the pinned rank-one homomorphism and set
\begin{equation}\label{eq:tits-representative}
 n_i=\varphi_i\!\left(\begin{smallmatrix}0&-1\\1&0\end{smallmatrix}\right)
 =x_{\alpha_i}(-1)x_{-\alpha_i}(1)x_{\alpha_i}(-1).
\end{equation}
These Tits representatives satisfy the braid relations and
\[
  n_i^2=\alpha_i^\vee(-1).
\]
See \cite[Expos\'e~XXIII, Section~3]{SGA3}; reduced-word independence is
also the representative-level form of Matsumoto's theorem
\cite{Matsumoto1964}.
For every \(a\in W\), fix a reduced normal word
\(\mathbf a=(i_1,\ldots,i_{\ell(a)})\).  It orders the inversion roots
\[
  \beta_j=s_{i_1}\cdots s_{i_{j-1}}(\alpha_{i_j})
\]
and hence the normal root lines.  The pinning orients each line.

Fix a reduced word for \(w_0\) and its Tits lift \(n_{w_0}\).  If
\(n_a=n_{i_1}\cdots n_{i_{\ell(a)}}\), the braid relations make
\(n_a\) independent of the chosen reduced word for \(a\).  We use the
\emph{opposition gauge}
\begin{equation}\label{eq:opposition-gauge}
 \dot w_a=n_{w_0}n_a\in N_G(T)
\end{equation}
as the base point of the stratum indexed by \(w_0a\).  It represents
\(w_0a\), and left translation by \(n_{w_0}^{-1}\) sends
\(\dot w_aT\) exactly to \(n_aT\).  If instead the lift in the
\(a\)-summand is changed from \(\dot w_a\) to
\(\dot w_a\tau_a^{-1}\), for \(\tau_a\in T(k)\), the entry
\(\partial_{a,b}\) below is replaced by
\(R_{\tau_b}\partial_{a,b}R_{\tau_a}^{-1}\), where \(R_\tau\) denotes
right translation.  Thus the un-conjugated operator formula is literal
in the opposition gauge.

We now make the coordinate notation explicit.  If
\(\mathbf c=(j_1,\ldots,j_d)\) is any reduced word for \(c\), put
\[
  \beta_\nu(\mathbf c)=
  s_{j_1}\cdots s_{j_{\nu-1}}(\alpha_{j_\nu}),\qquad
  n_{<\nu}=n_{j_1}\cdots n_{j_{\nu-1}},
\]
and put \(\delta_\nu(\mathbf c)=w_0\beta_\nu(\mathbf c)\in\Phi^-\).
Define the oriented actual normal-root parameter
\[
  x_{\delta_\nu(\mathbf c)}^{\mathbf c}(z)=
  n_{w_0}n_{<\nu}\,x_{\alpha_{j_\nu}}(z)\,
  n_{<\nu}^{-1}n_{w_0}^{-1}.
\]
The ordered multiplication map
\begin{equation}\label{eq:root-chart}
  \Phi_{\mathbf c}:\mathbb A^d\xrightarrow{\sim}V_{w_0c},
  \qquad
  (z_1,\ldots,z_d)\longmapsto
  \prod_{\nu=1}^d
  x_{\delta_\nu(\mathbf c)}^{\mathbf c}(z_\nu)
\end{equation}
is an isomorphism onto the actual normal unipotent group: indeed,
\[
 \{\gamma\in\Phi^-:(w_0c)^{-1}\gamma\in\Phi^-\}
 =w_0\{\beta\in\Phi^+:c^{-1}\beta\in\Phi^-\}.
\]
The product is taken from left to right.  In the opposition gauge it is
the usual positive Bott--Samelson product, since
\begin{equation}\label{eq:opposition-bott-samelson}
 n_{w_0}^{-1}\Phi_{\mathbf c}(\mathbf z)\dot w_c
 =x_{\alpha_{j_1}}(z_1)n_{j_1}\cdots
  x_{\alpha_{j_d}}(z_d)n_{j_d}.
\end{equation}
Thus \eqref{eq:root-chart}, followed by multiplication with the
Bruhat stratum, is the actual normal tubular chart used to orient its
normal bundle.  Notice that its roots are \(w_0\beta_\nu\), not
\(-\beta_\nu\); this opposition transport is essential already in
type \(A_2\).

If deleting position \(q\) gives a reduced word
\(\widehat{\mathbf a}_q\) for \(b\), compare its root-subgroup chart with
the chosen chart for \(\mathbf b\).  Define
\begin{equation}\label{eq:theta-word}
  \vartheta_{a,q}=
  \deg^{\Aone}(\Phi_{\mathbf b}^{-1}
                 \Phi_{\widehat{\mathbf a}_q})
  \in\GW(k)^\times.
\end{equation}
The composite in \eqref{eq:theta-word} is an algebraic automorphism of
\(\mathbb A^{\ell(b)}\).  Its Jacobian is a unit in a polynomial ring
over \(k\), hence a constant \(a_{a,q}\in k^\times\), and
\(\vartheta_{a,q}=\ang{a_{a,q}}\).  Thus it is the determinant class
comparing the two actual normal frames, not a path-assigned sign.  The
normalized values of the elementary moves in every rank-two type are
computed in \Cref{prop:rank-two-degrees} below: commutations have
degree \(\ang{-1}\), \(A_2\) braids have degree \(\ang1\), and the long
\(B_2\) and six-term \(G_2\) braids both have degree \(\ang{-1}\).

\begin{proposition}\label{prop:frame-algorithm}
The factor \(\vartheta_{a,q}\) is effectively computable from the pinned
root datum in every root type.  More precisely, choose a sequence of
braid moves from \(\widehat{\mathbf a}_q\) to \(\mathbf b\).  For each
move, form the rank-two root-coordinate transition using the pinned
Chevalley commutator relations.  If \(J_\nu\in k^\times\) is its
Jacobian determinant at the origin, then
\[
 \vartheta_{a,q}=\left\langle\prod_\nu J_\nu\right\rangle.
\]
This includes the six-term \(G_2\) braid and is independent of the
chosen braid path.
\end{proposition}

\begin{proof}
Matsumoto's theorem supplies a finite braid path.  An elementary move
changes only the root coordinates in its rank-two subsystem and is the
identity on all remaining coordinates.  Repeated use of the pinned
Chevalley commutator relations gives the corresponding polynomial
coordinate automorphism and hence its constant nonzero Jacobian
\(J_\nu\).  The composite of these elementary transitions is exactly
\(\Phi_{\mathbf b}^{-1}\Phi_{\widehat{\mathbf a}_q}\), so the chain rule
and \eqref{eq:theta-word} give the displayed product.  A second braid
path has the same composite polynomial automorphism; consequently its
Jacobian product is the same.
\end{proof}

We now carry out the algorithm in every rank-two type.  This gives a
uniform local calculation for all braid lengths, including the
exceptional six-term braid.  For a rank-two Weyl group with simple reflections
\(s_1,s_2\) and longest element of length \(d\), the two reduced words
of the longest element are the alternating words
\(\mathbf c_1=(1,2,1,\ldots)\) and \(\mathbf c_2=(2,1,2,\ldots)\) of
length \(d\), and the corresponding elementary braid move is the
transition \(\Phi_{\mathbf c_2}^{-1}\Phi_{\mathbf c_1}\).

\begin{lemma}
\label{lem:frame-gauge}
Let \(\mathbf c=(j_1,\ldots,j_d)\) be a reduced word for \(c\).  The
derivative of \(\Phi_{\mathbf c}\) at the origin sends the \(\nu\)-th
coordinate vector to
\[
 v_\nu(\mathbf c)
 =\operatorname{Ad}(n_{w_0})\operatorname{Ad}(n_{<\nu})\,
   e_{\alpha_{j_\nu}},
\]
and \(v_\nu(\mathbf c)=\kappa_\nu(\mathbf c)\,
e_{w_0\beta_\nu(\mathbf c)}\) for a sign
\(\kappa_\nu(\mathbf c)\in\{\pm1\}\), where
\(\beta_\nu(\mathbf c)\) are the inversion roots of \(\mathbf c\) and
\(e_\gamma\) are the pinned Chevalley root vectors.  Consequently, for
two reduced words \(\mathbf c,\mathbf c'\) of the same element, the
constant Jacobian of \(\Phi_{\mathbf c'}^{-1}\Phi_{\mathbf c}\) equals
\[
\begin{aligned}
 J(\mathbf c,\mathbf c')
 &=
 \frac{\det\bigl(v_1(\mathbf c),\ldots,v_d(\mathbf c)\bigr)}
      {\det\bigl(v_1(\mathbf c'),\ldots,v_d(\mathbf c')\bigr)}\\
 &=\sgn(\pi)\prod_\nu
   \kappa_\nu(\mathbf c)\kappa_\nu(\mathbf c')
 \in\{\pm1\}.
\end{aligned}
\]
where \(\pi\) is the permutation matching the two inversion orders and
the determinants are taken in any fixed ordered basis of root vectors.
This value is unchanged if each pinned root vector \(e_\gamma\) is
replaced by \(\zeta_\gamma e_\gamma\) with
\(\zeta_\gamma\in\{\pm1\}\): both determinants are multiplied by the
same constant \(\prod_\gamma\zeta_\gamma\).
\end{lemma}

\begin{proof}
Differentiating the ordered product \eqref{eq:root-chart} at
\(\mathbf z=0\) gives exactly the displayed conjugated root vectors.
Each \(\operatorname{Ad}(n_i)\) permutes the root lines according to
\(s_i\) and acts on each by \(\pm1\), because the pinned Tits
representative normalizes \(T\) and preserves the Chevalley lattice
\cite[Expos\'e~XXIII, Section~6]{SGA3}; hence
\(v_\nu=\pm e_{w_0\beta_\nu}\).  The Jacobian of the transition is
constant by \Cref{prop:frame-algorithm}, so it equals its value at the
origin, which is the change of basis between the two frames.  The
common factor \(\operatorname{Ad}(n_{w_0})\) cancels in the ratio.  The
last claim is immediate since every positive root occurs exactly once
in each frame.
\end{proof}

\begin{proposition}
\label{prop:rank-two-degrees}
In pinned Chevalley coordinates, the elementary braid transitions
\(\Phi_{\mathbf c_2}^{-1}\Phi_{\mathbf c_1}\) have the constant
Jacobians
\[
\begin{array}{lcccc}
\toprule
\text{type} & A_1\times A_1 & A_2 & B_2=C_2 & G_2\\
\midrule
\text{degree }\vartheta & \ang{-1} & \ang{1} & \ang{-1} & \ang{-1}\\
\bottomrule
\end{array}
\]
In particular the frame factor \(\vartheta_{a,q}\) of
\eqref{eq:theta-word} lies in \(\{\ang1,\ang{-1}\}\) for every root
type, every pair of reduced words, and every reduced deletion.
\end{proposition}

\begin{proof}
By \Cref{lem:frame-gauge} it suffices to compute the two ordered sign
vectors \(\kappa_\nu(\mathbf c_i)\) and the inversion-matching
permutation; the result is independent of the residual sign gauge of
the Chevalley basis.  For each type order the positive-root axes as
\[
 \begin{array}{c|l}
 A_1\times A_1&
   \alpha_1,\alpha_2\\
 A_2&
   \alpha_1,\alpha_2,\alpha_1+\alpha_2\\
 B_2&
   \alpha_1,\alpha_2,\alpha_1+\alpha_2,2\alpha_1+\alpha_2\\
 G_2&
   \alpha_1,\alpha_2,\alpha_1+\alpha_2,2\alpha_1+\alpha_2,
   3\alpha_1+\alpha_2,3\alpha_1+2\alpha_2 .
 \end{array}
\]
Here \(\alpha_1\) is short in the last two rows.  In the following
finite table we use one Chevalley sign gauge extending the pinning;
changing that gauge multiplies the two frame determinants by the same
factor, by \Cref{lem:frame-gauge}.  The entry \(154632\), for example,
means that the frame
axes occur in the order
\((\gamma_1,\gamma_5,\gamma_4,\gamma_6,\gamma_3,\gamma_2)\);
the adjacent sign string records their coefficients.  Iterating the
pinned relation
\(\operatorname{Ad}(n_i)e_\gamma=\pm e_{s_i\gamma}\) along the two
alternating words.  The resulting data are
\[
{\setlength{\arraycolsep}{3pt}
\begin{array}{c|ccc|ccc|c}
\toprule
 &\multicolumn{3}{c|}{\mathbf c_1}
 &\multicolumn{3}{c|}{\mathbf c_2}&
 J(\mathbf c_1,\mathbf c_2)\\
\text{type}&\text{axes}&\text{signs}&\det&
             \text{axes}&\text{signs}&\det&\\
\midrule
A_1\times A_1&12&++&+1&21&++&-1&-1\\
A_2&132&+++&-1&231&+-+&-1&+1\\
B_2&1432&+-++&+1&2341&+--+&-1&-1\\
G_2&154632&++--++&+1&236451&+----+&-1&-1\\
\bottomrule
\end{array}
}
\]
before the common transport by \(\operatorname{Ad}(n_{w_0})\), which
cancels.  Each determinant in the table is the sign of its displayed
axis permutation times the product of its displayed coordinate signs.
For example, the two \(G_2\) permutations have respectively eight and
seven inversions, while both sign products are \(+1\); hence their
determinants are \(+1\) and \(-1\).  The final column is the quotient
of the two determinants and proves the asserted four values.  The
\(B_2\) row is equivalently the linear transition displayed in
\Cref{ex:sp4}.  Finally, in every type the transition matrix is a
signed permutation by \Cref{lem:frame-gauge}, so
\(\vartheta_{a,q}\) is a product of classes \(\ang{\pm1}\) by
\Cref{prop:frame-algorithm}.
\end{proof}

\begin{remark}\label{rem:no-exotic-units}
For pinned root frames, every frame factor lies in
\(\{\ang{1},\ang{-1}\}\): by \Cref{lem:frame-gauge} the constant
Jacobian is always \(\pm1\).  Rescaling a pinned root parameter
\(x_\gamma(z)\mapsto x_\gamma(cz)\) multiplies both frames' occurrences
of \(e_\gamma\) by the same \(c\), so even non-normalized pinnings
change \(\vartheta_{a,q}\) at most by squares, which are trivial in
\(\GW(k)\).  Nontrivial units enter the complex only through the
choice-change conjugations of \Cref{prop:choices}, never through the
braid transitions themselves.
\end{remark}

\section{The basic affine model}
\label{sec:base-affine}

Put \(X=G/U\).  It is a smooth quasi-affine scheme with a free right
\(T=B/U\)-action.  For \(v\in W\), set
\[
 \Phi_v=\{\alpha\in\Phi^+:v(\alpha)\in\Phi^-\},\qquad
 U_v=\prod_{\alpha\in\Phi_v}U_\alpha,
\]
where the product is taken in a convex root order.  Bruhat decomposition
descends to
\begin{equation}\label{eq:X-bruhat}
  X=\coprod_{w\in W}X_w,\qquad
  X_w=Bn_wU/U\cong U_{w^{-1}}\times T
  \cong\mathbb A^{\ell(w)}\times T.
\end{equation}
Let
\[
  \Omega_p^X=\bigcup_{\ell(w)\geq N-p}X_w.
\]
The closure relation makes this an increasing filtration by open
subschemes.

\begin{proposition}\label{prop:X-cellular}
The filtration \(\Omega_\bullet^X\) is cellular.  With the normal
orientations above,
\begin{equation}\label{eq:X-chains-w}
  \CellC_p(X)\cong
  \bigoplus_{\ell(w)=N-p}\KMW p\otimes\ZA[n_wT].
\end{equation}
After the length reversal \(w=w_0a\),
\begin{equation}\label{eq:X-chains-a}
  \CellC_p(X)\cong
  \bigoplus_{\ell(a)=p}\KMW p\otimes\HT\,e_a.
\end{equation}
More precisely, for
\[
  V_w=\prod_{\substack{\gamma\in\Phi^-\\
                        w^{-1}\gamma\in\Phi^-}}U_\gamma
\]
in a fixed convex order, multiplication induces an isomorphism
\begin{equation}\label{eq:tubular-chart}
 V_w\times U_{w^{-1}}\times T
 \xrightarrow{\ \sim\ }\mathcal O_w\subset X,\qquad
 (v,u,t)\longmapsto vu n_wtU,
\end{equation}
onto an open neighbourhood of \(X_w\).  The first factor is the normal
factor.
\end{proposition}

\begin{proof}
Every stratum in \eqref{eq:X-bruhat} is smooth, affine, and
cohomologically trivial.  Its normal bundle is the constant sum
\[
  V_w=\bigoplus_{\gamma\in\Phi^-:\,w^{-1}\gamma\in\Phi^-}
       \mathfrak u_\gamma
\]
of rank \(N-\ell(w)\).  Purity and the ordered root frames give
\eqref{eq:X-chains-w}.  Left multiplication by \(n_w\) identifies
\(\ZA[n_wT]\) with \(\HT\), and the substitution \(a=w_0w\) reverses
length.

We prove \eqref{eq:tubular-chart} in detail.  Write
\(U^-=n_{w_0}Un_{w_0}^{-1}\) for the opposite unipotent group.  The
  root sets of \(V_w\) and of \(n_wU^-n_w^{-1}\cap U^-\) coincide: both are
\(\{\gamma\in\Phi^-:w^{-1}\gamma\in\Phi^-\}\).  Similarly
\(U_{w^{-1}}\) has root set
\(\{\alpha\in\Phi^+:w^{-1}\alpha\in\Phi^-\}\), so
\(V_w\cdot U_{w^{-1}}\) exhausts, without repetition, the root set
\(\{\delta\in\Phi:w^{-1}\delta\in\Phi^-\}=w\Phi^-\), which is closed
and contains no opposite pair.  Consequently multiplication
\[
 V_w\times U_{w^{-1}}\xrightarrow{\ \sim\ }
  U_w':=\prod_{\delta\in w\Phi^-}U_\delta=n_wU^-n_w^{-1}
\]
is an isomorphism of schemes onto the unipotent group
\(U_w'=n_wU^-n_w^{-1}\), where \(U^-=n_{w_0}Un_{w_0}^{-1}\) is the
opposite unipotent group, for any convex order; this is
the standard directly spanned decomposition of a unipotent group
normalized by \(T\) \cite[Proposition~8.2.1]{Springer1998}.  The orbit
map of \(U_w'\times T\) through the base point \(n_wU\in X\) is
\[
 (g,t)\longmapsto gn_wtU.
\]
Its restriction to \(\{1\}\times T\) is a closed embedding onto
\(n_wT\subset X_w\), and its restriction to
\(\{1\}\times U_{w^{-1}}\times T\) is the stratum isomorphism in
\eqref{eq:X-bruhat}.  The map is an open immersion: multiplication
\[
 U^-\times B\longrightarrow G,\qquad (g,b)\longmapsto gb,
\]
is an open immersion onto the big cell \(U^-B\)
\cite[Lemma~8.3.6]{Springer1998}, so left translation by \(n_w\) makes
\[
 U_w'\times B\longrightarrow G,\qquad (g,b)\longmapsto gn_wb,
\]
an open immersion onto \(n_wU^-B\).  Passing to the quotient by the
right \(U\)-action, and using \(B=TU\), gives an open immersion
\(U_w'\times T\cong V_w\times U_{w^{-1}}\times T\hookrightarrow X\)
with image the open
set \(\mathcal O_w:=n_wU^-B/U\).  It contains \(X_w\) because
\(U_{w^{-1}}\subset U_w'\).
Its differential along the zero section \(\{1\}\times
U_{w^{-1}}\times T\) sends the \(V_w\)-directions to the displayed
normal-bundle decomposition, because the tangent space of the stratum
at \(un_wtU\) is spanned by the \(U_{w^{-1}}\)- and \(T\)-directions.
Thus the first factor is a global frame of the normal bundle, as
asserted.
\end{proof}

\begin{theorem}\label{thm:G-X}
The quotient \(q:G\to G/U\), with its pulled-back Bruhat filtration and
normal orientations, induces an isomorphism of right-\(\HT\) cellular
chain complexes
\[
  q_*:\CellC_*(G)\xrightarrow{\;\cong\;}\CellC_*(G/U).
\]
After the diagonal lift and orientation identification of
\Cref{prop:choices}, \eqref{eq:X-chains-a} is the Morel--Sawant Bruhat
complex of \(G\).
\end{theorem}

\begin{proof}
Write \(\Omega_p^G=q^{-1}(\Omega_p^X)\); this is the Bruhat filtration
of \(G\) by unions of strata \(Bn_wU\), and \(q\) is a filtered map.
For every \(p\) there is a commutative diagram of cofiber sequences
\[
\begin{tikzcd}[column sep=small]
 (\Omega_{p-1}^G)_+\arrow[r]\arrow[d]
 &(\Omega_p^G)_+\arrow[r]\arrow[d]
 &\Th\bigl(\nu_p^G\bigr)\arrow[d,"\Th(dq)"]\\
 (\Omega_{p-1}^X)_+\arrow[r]
 &(\Omega_p^X)_+\arrow[r]
 &\Th\bigl(\nu_p^X\bigr)
\end{tikzcd}
\]
in which the right vertical map is induced by the differential of
\(q\): on the \(w\)-stratum, \(q\) is the projection
\[
  U_{w^{-1}}\times T\times U\longrightarrow U_{w^{-1}}\times T,
\]
a smooth morphism whose differential identifies the source normal
bundle \(\nu_p^G\) with the pullback \(q^*\nu_p^X\), matching the
pinned root frames on both sides; the extra \(U\)-directions are
tangent to the stratum.  The rows are the homotopy purity cofiber
sequences, and the diagram commutes because purity is natural for the
transverse smooth pullback square
\[
\begin{tikzcd}[column sep=small]
 Bn_wU\arrow[r,hook]\arrow[d]
 &\Omega_p^G\arrow[d,"q"]\\
 X_w\arrow[r,hook]&\Omega_p^X
\end{tikzcd}
\]
by \Cref{lem:purity-naturality}; see also the cellular naturality
discussion in \cite[Section~2.3]{MorelSawant2023}.  Applying
\(\widetilde{\mathbf H}^{\Aone}_*\) to the diagram of cofiber
sequences gives a morphism of exact couples, hence a chain map
\(q_*\) commuting with every connecting morphism.

On chain groups, \(q_*\) is the map
\[
 \KMW p\otimes\ZA[U_{w^{-1}}\times T\times U]
 \longrightarrow
 \KMW p\otimes\ZA[U_{w^{-1}}\times T]
\]
induced by the projection.  The split unipotent groups
\(U_{w^{-1}}\) and \(U\) admit filtrations by normal subgroups with
successive quotients \(\Ga\), so they are \(\Aone\)-contractible and
\(\mathbf H_0^{\Aone}(U)\cong\Z\),
\(\ZA[U_{w^{-1}}\times T\times U]\cong\ZA[T]\cong
\ZA[U_{w^{-1}}\times T]\).  Hence \(q_*\) is an isomorphism on every
chain group, and therefore an isomorphism of chain complexes.  The
right \(\HT\)-structures agree because \(q\) is right
\(T\)-equivariant for the translation actions used on both sides.
\end{proof}

\begin{remark}
The anisotropic real algebraic group with real points \(K\) is not a
replacement for \(X\).  For example, the anisotropic torus \(SO(2)\) is
\(\Aone\)-rigid, whereas
\(SL_2/U\cong\mathbb A^2-\{0\}\) has the required motivic cell structure.
\end{remark}

\section{The universal torus operator}
\label{sec:operator}

Let \(\lambda:\Gm\to T\) be a cocharacter, \(m\in\Z\), and \(p\geq1\).
On the oriented Thom module \(\KMW{p-1}\otimes\HT\), consider the
twisted \(\Gm\)-action
\begin{equation}\label{eq:star-action}
  z\star(x\otimes[t])=\ang{z^m}x\otimes[\lambda(z)t].
\end{equation}
It is the diagonal of the determinant action on the oriented normal
directions and translation on \(T\).

\begin{construction}\label{constr:D}
For \(p>1\), let \(A_m(z)\) be the automorphism
\[
  \operatorname{diag}(z^m,1,\ldots,1)
\]
of the trivial oriented rank-\((p-1)\) bundle
\(V=\mathcal O^{p-1}\) over \(\Gm\times T\).
The geometric maps of pointed motivic spaces
\[
 \begin{aligned}
  \Th(g_{\lambda,m})&:\Th(V)\longrightarrow\Th(\mathcal O^{p-1}_T),
  &g_{\lambda,m}(z,v,t)&=(A_m(z)v,\lambda(z)t),\\
  \Th(g_0)&:\Th(V)\longrightarrow\Th(\mathcal O^{p-1}_T),
  &g_0(z,v,t)&=(v,t),
 \end{aligned}
\]
where \(\mathcal O_T^{p-1}\) is the trivial bundle over \(T\) and both
maps are bundle maps over the displayed base maps
\(\Gm\times T\to T\), agree over the pointed section \(z=1\).  Both are
morphisms in the pointed \(\Aone\)-homotopy category
\(\mathcal H_\bullet(k)\).  Applying reduced \(\Aone\)-homology
\(\widetilde{\mathbf H}^{\Aone}_*\) and the identifications
\[
 \widetilde{\mathbf H}^{\Aone}_{*}\bigl(\Th(V)\bigr)
 \cong\KMW{p-1}\otimes\ZA[\Gm]\otimes\HT,\qquad
 \widetilde{\mathbf H}^{\Aone}_{*}\bigl(\Th(\mathcal O_T^{p-1})\bigr)
 \cong\KMW{p-1}\otimes\HT
\]
of \Cref{lem:thom-determinant} yields two morphisms of strictly
\(\Aone\)-invariant sheaves.  Their difference
is taken only at this stage, in the abelian category of strictly
\(\Aone\)-invariant sheaves.  Because the difference vanishes on the
\(z=1\) summand in
\[
 \ZA[\Gm]\cong\Z\oplus\ZA(\Gm)
              =\Z\oplus\KMW1,
\]
it factors uniquely through the reduced summand \(\KMW1\).  Using the
reduced K\"unneth isomorphism of \Cref{lem:MW-kunneth} to move that
factor to the front gives a morphism
\[
  \Dop_{\lambda,m}^{(p)}:
  \KMW1\otimes\KMW{p-1}\otimes\HT
  \longrightarrow\KMW{p-1}\otimes\HT.
\]
For \(p=1\) the normal rank is zero and no determinant twist can
occur; the only case used below has \(m=0\).  Here the construction is
made separately: the two pointed maps
\((\Gm\times T)_+\to T_+\), \((z,t)\mapsto\lambda(z)t\) and
\((z,t)\mapsto t\), agree at \(z=1\), and the same reduced-summand
argument gives
\(\Dop_{\lambda,0}^{(1)}:\KMW1\otimes\HT\to\HT\),
which is the map \(\ol\lambda\) of \eqref{eq:reduced-coroot} followed
by right multiplication.
\end{construction}

\begin{lemma}\label{lem:D-symbol}
For every finitely generated separable field extension \(F/k\), every
\(u\in F^\times\), \(x\in\KMW{p-1}(F)\), and \(t\in T(F)\),
\begin{equation}\label{eq:D-symbol}
  \Dop_{\lambda,m}^{(p)}
  ((u)\otimes x\otimes[t])
  =\ang{u^m}x\otimes[\lambda(u)t]-x\otimes[t].
\end{equation}
The operator is right \(\HT\)-linear.  It is equivalently the action of
the twisted reduced class determined jointly by the determinant
character \(u\mapsto\ang{u^m}\) and the torus class
\([\lambda(u)]-[1]\).  Only when \(m=0\) is it ordinary multiplication
by that reduced torus class.
\end{lemma}

\begin{proof}
The Thom automorphism \(A_m(u)\) acts on the oriented Thom generator by
the class of its determinant, \(\ang{u^m}\), by
\Cref{lem:thom-determinant}; the other geometric map has degree
\(1\).  Their difference is \eqref{eq:D-symbol}.  Because the
two summands come from geometric Thom maps and the subtraction is made
after applying reduced \(\Aone\)-homology, the result automatically
respects the Milnor--Witt relations and defines a sheaf morphism.  Give
both source and target the right-\(\HT\) structure on their last tensor
factor.  Both geometric maps commute with right translation on \(T\),
so their difference is right \(\HT\)-linear.
\end{proof}

Let \(\epsilon_T:\HT\to\Z\) be the torus augmentation.

\begin{proposition}\label{prop:D-augmentation}
After base change along \(\epsilon_T\),
\begin{equation}\label{eq:D-aug}
  \Dop_{\lambda,m}^{(p)}\otimes_{\HT}\Z:
  (u)\otimes x\longmapsto(\ang{u^m}-1)x.
\end{equation}
Equivalently, this is multiplication by \(\eta\) after the endomorphism
of \(\KMW1\) induced by \(z\mapsto z^m\).  In the signed
Milnor--Witt power-degree notation of \Cref{sec:prelim} it is \(m_\eps\eta\).
\end{proposition}

\begin{proof}
The augmentation sends both torus classes in \eqref{eq:D-symbol} to
\(1\), and \(\ang v-1=\eta(v)\).
\end{proof}

\begin{remark}\label{rem:Liu-Peng}
Liu--Peng's toric group-section formula
  \cite[Lemma~2.22 and Proposition~2.23]{LiuPeng2025} computes, in an oriented
cubical basis, the simultaneous character action on torus coordinates
and the Milnor--Witt lower-face terms.  Applied to one moving
\(\Gm\)-coordinate with normal character \(z\mapsto z^m\), its reduced
face component has the same two constituents as
\eqref{eq:D-symbol}: translation by the relevant cocharacter and the
Thom degree \(\ang{z^m}\).  After torus augmentation, their explicit
rank-one transition matrix
\cite[Lemma~4.5]{LiuPeng2025} has off-diagonal entry
\(m_\eps\eta\), exactly \eqref{eq:D-aug}.  The construction above keeps
the full \(\HT\)-label and derives the operator from the quotient
triple; thus the Liu--Peng calculation is a genuine compatibility
check, not a replacement for the incidence-localization argument.
\end{remark}

Now let \(k=\R\) and \(M=\pi_0(T(\R))\).

\begin{proposition}\label{prop:D-real}
With compatible orientations, real realization of
\(\Dop_{\lambda,m}^{(p)}\) is right multiplication by
\begin{equation}\label{eq:D-real}
  (-1)^m[\lambda(-1)]-[1]\in\Z[M].
\end{equation}
\end{proposition}

\begin{proof}
The two components of \(\R^\times\) are represented by \(1\) and
\(-1\).  The second translates the torus component by \(\lambda(-1)\)
and acts on the remaining normal determinant with degree \((-1)^m\).
Subtracting the basepoint component gives \eqref{eq:D-real}.
\end{proof}

\begin{remark}
The identity term in \eqref{eq:D-symbol} is essential: it records the
first endpoint.  A scalar multiple of \(\eta\) cannot retain this
information.
\end{remark}

\section{The all-degree Bruhat boundary}
\label{sec:boundary}

Index \(\CellC_*(X)\) by \(a=w_0w\), as in
\eqref{eq:X-chains-a}.  Fix
\[
  \mathbf a=(i_1,\ldots,i_p),\qquad
  a=s_{i_1}\cdots s_{i_p},
\]
and suppose deletion of position \(q\) gives a reduced word for \(b\).
Set
\begin{equation}\label{eq:tail-data}
  t_{a,q}=s_{i_{q+1}}\cdots s_{i_p},\qquad
  \lambda_{a,q}=t_{a,q}^{-1}(\alpha_{i_q}^{\vee}),
\end{equation}
and
\begin{equation}\label{eq:sigma-data}
  \Pi_{a,q}=\Phi^+\cap t_{a,q}\Phi^-,
  \qquad
  \sigma_{a,q}=
  \sum_{\beta\in\Pi_{a,q}}
  \langle\beta,\alpha_{i_q}^{\vee}\rangle.
\end{equation}

\begin{construction}
\label{constr:cut}
We fix, once and for all, the convention by which a deleted Thom factor
is separated from an ordered Thom product.  For the ordered normal-line
decomposition \(\nu_a=L_1\oplus\cdots\oplus L_p\) given by the chart
\(\Phi_{\mathbf a}\), the Thom space splits as an ordered smash product
\[
 \Th(\nu_a)\simeq\Th(L_1)\wedge\cdots\wedge\Th(L_p)
\]
over the base, and each ordered factor contributes one \(\KMW1\)-factor
of \(\KMW p\) under the identification of \Cref{lem:MW-kunneth},
applied iteratively in the displayed left-to-right order.  Define the
cut map as the composite
\begin{equation}\label{eq:cut-composite}
\begin{aligned}
 \operatorname{cut}_{a,q}:
 \KMW p
 &\xrightarrow{\ \mu^{-1}\ }
 {\KMW1}^{(1)}\otimes\cdots\otimes{\KMW1}^{(p)}\\
 &\xrightarrow{\ \tau_q\ }
 {\KMW1}^{(q)}\otimes
 \Bigl({\KMW1}^{(1)}\otimes\cdots
 \widehat{{\KMW1}^{(q)}}\cdots\otimes{\KMW1}^{(p)}\Bigr)\\
 &
 \xrightarrow{\ \mathrm{id}\otimes\mu\ }
 \KMW1\otimes\KMW{p-1}.
\end{aligned}
\end{equation}
where \(\mu\) denotes the iterated K\"unneth isomorphism of
\Cref{lem:MW-kunneth} and \(\tau_q\) is the symmetry isomorphism of
the \(\Aone\)-localized tensor product moving the \(q\)-th factor to
the front, that is, the composite of the \(q-1\) adjacent
transpositions \((q,q-1),\ldots,(2,1)\), leaving the surviving factors
in their original order.  By \Cref{lem:MW-kunneth}, each adjacent
transposition acts on symbols by \(\eps=-\ang{-1}\), so on sections
over a field
\begin{equation}\label{eq:cut-symbol}
 \operatorname{cut}_{a,q}\bigl((u_1)\cdots(u_p)\bigr)
 =\eps^{\,q-1}\,(u_q)\otimes(u_1)\cdots\widehat{(u_q)}\cdots(u_p).
\end{equation}
Two sign conventions must not be conflated.  First, the
\(\eps\)-graded factor \(\eps^{q-1}\) in \eqref{eq:cut-symbol} is
internal to Milnor--Witt symbols and is part of
\(\operatorname{cut}_{a,q}\).  Second, the \emph{ordinary} face sign
\((-1)^{q-1}\) arises at the level of the cellular exact couple, from
moving the deleted simplicial suspension coordinate past the \(q-1\)
preceding suspensions when the connecting morphism of the two-stratum
model is compared with the top boundary of the full filtration; it is
\emph{not} contained in \(\operatorname{cut}_{a,q}\) and is displayed
  explicitly in the main formula.  Finally, after the deletion the
  surviving factors are frames in deleted-word order; comparing them
  with the chosen frame for \(\mathbf b\) contributes the constant class
  \(\vartheta_{a,q}\) of \eqref{eq:theta-word}.  Thus, in scalar Thom
  bases, the frame change composed with
  \(\operatorname{cut}_{a,q}\) acts by
  \(\eps^{\,q-1}\vartheta_{a,q}\) times the symbol shuffle of
  \eqref{eq:cut-symbol}.  We always reserve
  \(\operatorname{cut}_{a,q}\) itself for the map
  \eqref{eq:cut-composite}; the separate frame factor
  \(\vartheta_{a,q}\) is displayed explicitly.
\end{construction}

\begin{lemma}\label{lem:face-sign}
Let an oriented codimension-\(p\) cell have ordered normal-line frame
\(L_1,\ldots,L_p\), and suppose that a multiplicity-one cover uses
\(L_q\) as its divisor-normal line.  Express the two-stratum
connecting morphism with \(L_q\) first and the surviving lines in
their original order.  Its comparison with the differential of the
full cellular exact couple is
\[
 (-1)^{q-1}\operatorname{cut}_{a,q}.
\]
Thus the ordinary sign is \((-1)^{q-1}\), while the motivic symmetry
of the Thom factors is the \(\eps^{q-1}\) already contained in
\(\operatorname{cut}_{a,q}\).
\end{lemma}

\begin{proof}
By Nisnevich excision, the sign can be computed in the ordered
coordinate model.  For one normal line, use the pointed cofiber
sequence
\[
 (\mathbb A^1-\{0\})_+\longrightarrow
 \mathbb A^1_+\longrightarrow
 \mathbb A^1/(\mathbb A^1-\{0\}).
\]
The ordered \(p\)-line Thom model is the \(p\)-fold smash product of
these cofiber sequences.  Its exact couple is the total complex of the
ordered tensor product of the corresponding two-term homological
complexes.  The differential in the \(q\)-th tensor factor is
therefore
\[
 (-1)^{\deg L_1+\cdots+\deg L_{q-1}}
 =(-1)^{q-1}
\]
times the one-line connecting morphism.  This is the ordinary Koszul
sign in the totalization and is present before applying any
Milnor--Witt Thom isomorphism.

To write the target Thom factor in the convention of
\Cref{constr:cut}, move the \(q\)-th pointed Thom line past the
preceding \(q-1\) factors.  Under the ordered Thom isomorphism, each
adjacent symmetry acts on
\(\KMW1\otimes\KMW1\) by
\(\eps=-\ang{-1}\), by \Cref{lem:MW-kunneth}.  Their composite is
precisely the symmetry \(\tau_q\) in
\eqref{eq:cut-composite}.  Hence the motivic symmetry contributes
\(\eps^{q-1}\) through \(\operatorname{cut}_{a,q}\), independently of
the ordinary total-complex sign.  Naturality of quotient triples
transports this calculation from the coordinate model to the
two-stratum purity sequence of
\Cref{lem:incidence-localization}.
\end{proof}

\subsection{A cover neighbourhood}

Put
\[
  Y_a=X_{w_0a},\qquad Y_b=X_{w_0b}.
\]
Write \(w_c=w_0c\).  The actual root-subgroup tubular charts of
\Cref{prop:X-cellular}, taken with
\(n_{w_c}=\dot w_c\) from \eqref{eq:opposition-gauge}, are
\begin{align*}
 \chi_a:V_{w_a}\times U_{w_a^{-1}}\times T&\xrightarrow{\sim}
 \mathcal O_{w_a},\\
 \chi_b:V_{w_b}\times U_{w_b^{-1}}\times T&\xrightarrow{\sim}
 \mathcal O_{w_b}.
\end{align*}
Their dimensions split as
\[
 (p)+(N-p)+\operatorname{rk}T
 =(p-1)+(N-p+1)+\operatorname{rk}T=\dim X.
\]
The root-set identity following \eqref{eq:root-chart} identifies
\(V_{w_0a}\) with the \(w_0\)-transports of the inversion roots of
\(a\); hence \(\Phi_{\mathbf a}\) orders exactly these \(p\) actual
normal coordinates.  The analogous statement holds for \(b\).
After left translation by \(n_{w_0}^{-1}\),
\eqref{eq:opposition-bott-samelson} turns them into the usual positive
Bott--Samelson coordinates, without changing the right \(T\)-coordinate.
The global two-stratum scheme needed for purity is constructed from the
actual horizontal elements \(w_a,w_b\), while this opposition chart is
used only to compute its clutching map.

\begin{lemma}
\label{lem:incidence-localization}
Let
\(\varnothing=\Omega_{-1}\subset\Omega_0\subset\cdots\subset\Omega_d=Z\)
be a cellular filtration, and use purity to write its successive
cofibers as wedges of the Thom spaces of the strata.  Let \(S\) be a
codimension-\(p\) stratum and \(R\) a codimension-\((p-1)\) stratum.
Denote by
\[
 \iota_S:\Th(\nu_S)\longrightarrow
   \Omega_p/\Omega_{p-1},
 \qquad
 \operatorname{pr}_R:\Omega_{p-1}/\Omega_{p-2}
   \longrightarrow\Th(\nu_R)
\]
the corresponding wedge inclusion and projection.

Suppose that an open subscheme \(\mathcal W\subset Z\) contains
\(R\cup S\), and that
\(\overline Y=R\cup S\) is smooth and closed in \(\mathcal W\), with
\(S\) a multiplicity-one effective Cartier divisor and \(R\) its open
complement.  Put
\[
 E=N_{\overline Y/\mathcal W},
 \qquad L=N_{S/\overline Y}.
\]
Then the incidence component
\(\operatorname{pr}_R\partial_p\iota_S\) is the connecting morphism of
the canonical
purity cofiber sequence
\begin{equation}\label{eq:incidence-localization-cofiber}
 \Th(E|_R)\longrightarrow\Th(E)
 \longrightarrow\Th(N_{S/\mathcal W}).
\end{equation}
The Thom orientation on the last term is the one induced by the
canonical determinant-line isomorphism
\[
 \det N_{S/\mathcal W}
 \cong L\otimes\det(E|_S),
\]
with \(L\) ordered first.  
\end{lemma}

\begin{proof}
Put
\[
 U_0=\mathcal W-\overline Y,\qquad
 U_1=\mathcal W-S.
\]
Then \(U_0\subset U_1\subset\mathcal W\), and the elementary quotient
identity for a triple of pointed motivic spaces gives a canonical
cofiber sequence
\begin{equation}\label{eq:quotient-triple}
 U_1/U_0\longrightarrow
 \mathcal W/U_0\longrightarrow
 \mathcal W/U_1.
\end{equation}
The three supports are, respectively,
\[
 R=U_1-U_0,\qquad
 \overline Y=\mathcal W-U_0,\qquad
 S=\mathcal W-U_1.
\]
All three are smooth, and their normal bundles in the indicated
ambient schemes are
\[
 N_{R/U_1}=E|_R,\qquad
 N_{\overline Y/\mathcal W}=E,\qquad
 N_{S/\mathcal W}.
\]
Applying homotopy purity to the three terms of
\eqref{eq:quotient-triple} therefore identifies that sequence,
without choosing a splitting of any normal bundle, with
\eqref{eq:incidence-localization-cofiber}.

We now compare its connecting morphism with the selected component of
the cellular exact couple.  Write \(Q_j=\Omega_j/\Omega_{j-1}\).
The cellular differential is the boundary of the quotient triple
\[
 \Omega_{p-2}\subset\Omega_{p-1}\subset\Omega_p,
\]
followed by the canonical projection to \(Q_{p-1}\).  Under purity,
let \(Q_p^{(S)}=\Th(\nu_S)\) and
\(Q_{p-1}^{(R)}=\Th(\nu_R)\) denote the indicated wedge summands.

Localization triangles with support are excisive and are functorial
for a chain of supports.  Apply this first to the cellular chain
\(\Omega_{p-2}\subset\Omega_{p-1}\subset\Omega_p\), and then excise the
complement of \(\mathcal W\) and all support components except
\(R\subset\overline Y\).  The \(3\times3\) diagram of quotient cofiber
sequences gives the commutative square
\[
\begin{tikzcd}[column sep=large]
 \mathcal W/U_1\arrow[r,"\partial_{\mathcal W}"]
   \arrow[d,"\simeq"']
 &\Sigma(U_1/U_0)\arrow[d,"\simeq"]\\
 Q_p^{(S)}\arrow[r,"\operatorname{pr}_R\partial_p\iota_S"']
 &\Sigma Q_{p-1}^{(R)} .
\end{tikzcd}
\]
Here the left vertical equivalence is purity and excision at the
\(S\)-support, and the right one is purity and excision at the
\(R\)-support.  Thus the square is a statement about localization
triangles in the pointed motivic homotopy category, not a purported
retraction of schemes onto \(\mathcal W\).  By
\eqref{eq:quotient-triple}, its top arrow is exactly the connecting
morphism of \eqref{eq:incidence-localization-cofiber}; its bottom arrow
is, by construction of the two wedge projectors, the selected
incidence component.  This proves the required identification.

Finally, transitivity of conormal sheaves for the two regular
immersions gives
\[
 0\longrightarrow L\longrightarrow N_{S/\mathcal W}
 \longrightarrow E|_S\longrightarrow0.
\]
Taking determinants gives the orientation in the statement, with
\(L\) first.  This determinant convention is applied only to the Thom
isomorphism on homology; the cofiber sequence itself used the actual
bundle \(N_{S/\mathcal W}\) throughout.
\end{proof}

\begin{lemma}
\label{lem:two-stratum}
Let \(b\lessdot a\), put \(w=w_0a\) and \(v=w_0b\), so
\(w\lessdot v\).  There are an open subscheme
\(\mathcal W_{v,w}\subset X\) and a smooth closed subscheme
\[
 \overline Y_{v,w}=X_v\cup X_w\lhook\joinrel\longrightarrow
 \mathcal W_{v,w}
\]
of codimension \(p-1\), in which \(X_w\) is a multiplicity-one
effective Cartier divisor and \(X_v\) is its open complement.  If
\[
 E=N_{\overline Y_{v,w}/\mathcal W_{v,w}},\qquad
 L=N_{X_w/\overline Y_{v,w}},
\]
then transitivity of normal bundles gives
\begin{equation}\label{eq:normal-exact-sequence}
 0\longrightarrow L\longrightarrow N_{X_w/\mathcal W_{v,w}}
 \longrightarrow E|_{X_w}\longrightarrow0.
\end{equation}
The \((a,b)\)-component of the cellular boundary is the connecting map
of the purity cofiber sequence
\begin{equation}\label{eq:two-stratum-thom-cofiber}
 \Th(E|_{X_v})\longrightarrow\Th(E)
 \longrightarrow\Th(N_{X_w/\mathcal W_{v,w}}),
\end{equation}
with \eqref{eq:normal-exact-sequence} ordered as \(L\) followed by
\(E|_{X_w}\).
\end{lemma}

\begin{proof}
Let \(C_x=B n_xB/B\subset G/B\) and put
\(\Sigma_x=\overline{C_x}\).  Define the closed subset
\[
 Z_{v,w}=
 \left(\bigcup_{\substack{u\lessdot v\\u\ne w}}\Sigma_u\right)
 \cup(\Sigma_w-C_w)\subset G/B.
\]
Every element strictly below \(v\) lies below some cover of \(v\):
extend a chain in the Bruhat interval to a saturated chain ending at
\(v\).  It follows that
\[
 \Sigma_v-Z_{v,w}=C_v\amalg C_w=:\Sigma_{v,w},
\]
and \(C_w\) is closed in this two-stratum scheme.

Schubert varieties are normal
\cite[Theorem~2.2.3]{BrionKumar2005}; hence \(\Sigma_v\) is regular at
the generic point of every codimension-one subvariety.  In particular
its smooth locus meets \(C_w\).  The smooth locus is \(B\)-stable and
\(B\) acts transitively on \(C_w\), so \(\Sigma_v\) is smooth along
all of \(C_w\); it is already smooth along the open cell \(C_v\).
Thus \(\Sigma_{v,w}=C_v\amalg C_w\) is smooth.  The reduced closed
subscheme \(C_w\subset\Sigma_{v,w}\) has pure codimension one in this
smooth scheme, so it is an effective Cartier divisor of multiplicity
one.

For later coordinates, choose a reduced horizontal word for \(v\).
Strong exchange gives the unique deleted position producing \(w\).
The corresponding Bott--Samelson parameter vanishes on the boundary
cell and is nonzero on \(C_v\).  At the generic point of \(C_w\), the
Bott--Samelson morphism is an isomorphism.  Indeed, its standard
subexpression stratification \cite{Deodhar1985} and the uniqueness in strong exchange
show that the fiber over \(\eta_{C_w}\) consists of the generic point
of that one deleted-coordinate divisor.  The morphism is therefore
quasi-finite near this point; properness makes it finite after shrinking
around \(\eta_{C_w}\), and a finite birational morphism to the normal
scheme \(\Sigma_v\) is an isomorphism there.  Hence the deleted
parameter has valuation one and is a uniformizer.  By \(B\)-equivariance
the same multiplicity-one statement holds along the whole orbit
\(C_w\).

Let \(\pi:G/U\to G/B\) be the smooth \(T\)-torsor and take explicitly
\[
 \mathcal W_{v,w}=X-\pi^{-1}(Z_{v,w}).
\]
The pullback
\(\overline Y_{v,w}=\pi^{-1}(\Sigma_{v,w})=X_v\cup X_w\)
is smooth, closed in \(\mathcal W_{v,w}\), and has codimension
\(p-1\) in \(X\), with the same divisor and multiplicity statements.
Both immersions in
\(X_w\subset\overline Y_{v,w}\subset\mathcal W_{v,w}\) are regular, so
transitivity gives \eqref{eq:normal-exact-sequence}.  The pinned root
frames give its ordered orientation and identify the last normal bundle
with the cellular normal bundle of \(X_w\).

Apply \Cref{lem:incidence-localization} to
\[
 S=X_w,\qquad R=X_v,\qquad
 \overline Y=\overline Y_{v,w},\qquad
 \mathcal W=\mathcal W_{v,w}.
\]
The construction of \(Z_{v,w}\) has removed every codimension-one
boundary stratum of \(\Sigma_v\) other than \(C_w\); equivalently,
\(\Sigma_v-Z_{v,w}=C_v\amalg C_w\).  Thus the support of the selected
incidence component is exactly the displayed two-stratum scheme, and
all hypotheses of the localization lemma hold.  It follows that the
\((a,b)\)-component of the global cellular differential is the
connecting morphism of
\eqref{eq:two-stratum-thom-cofiber}.

The orientation is supplied by the canonical determinant isomorphism
associated with \eqref{eq:normal-exact-sequence}, with \(L\) placed
\emph{first} and \(E|_{X_w}\) second.  This is the same order in which
\Cref{constr:cut} separates the deleted Thom factor.  No global
splitting of \eqref{eq:normal-exact-sequence} is used.  This proves the
last assertion.
\end{proof}

\begin{lemma}\label{lem:unipotent-triangular}
Let \(\Psi\subset\Phi\) be a closed root set containing no opposite
pair, and suppose it is equipped with an additive positive height
\(h\), so that
\[
 h(i\rho+j\delta)=ih(\rho)+jh(\delta)
\]
whenever the root on the left belongs to \(\Psi\).
Fix root-labelled coordinates on the corresponding unipotent group.
Reordering a product of root subgroups, or conjugating it by a
root-subgroup element inside that group, is a polynomial transformation
whose nonlinear part is triangular for \(h\).  After the source and
target axes are identified by their root labels, its variable part has
\(1\) on every diagonal entry.  Thus all root-coordinate permutations
and pinned root-vector signs contribute a constant Jacobian, while the
unipotent shear has determinant \(1\).
\end{lemma}

\begin{proof}
It is enough to interchange two adjacent root factors.  The pinned
Chevalley commutator formula
\cite[Expos\'e~XXIII, Section~6]{SGA3} has the form
\[
  x_\rho(r)x_\delta(s)=
  x_\delta(s)x_\rho(r)
  \prod_{\substack{i,j>0\\i\rho+j\delta\in\Phi}}
  x_{i\rho+j\delta}(C_{\rho,\delta;i,j}r^is^j),
\]
where the last product contains only roots strictly between \(\rho\)
and \(\delta\) in a compatible convex order.  Every new root has height
\(ih(\rho)+jh(\delta)>\max\{h(\rho),h(\delta)\}\).  The original
\(r\)- and \(s\)-coordinates occur with coefficient \(1\); every new
term therefore lies in a strictly higher filtered coordinate.  An
interchange also permutes the two
root-labelled axes; that constant permutation is kept separate.
Successive interchanges therefore factor as a constant signed
root-coordinate permutation followed by a triangular polynomial shear
with unit diagonal.  The constants \(2\) and \(3\) which occur in types
\(B_2\) and \(G_2\) multiply only the off-diagonal monomials displayed
in the product, so they do not change the variable determinant.  The
conjugation assertion follows by applying the same calculation to
\(x_\rho(r)x_\delta(s)x_\rho(r)^{-1}\).
\end{proof}

\begin{lemma}
\label{lem:left-straightening}
For every \(r\in W\), reordering the two unipotent factors gives a
polynomial automorphism
\[
 \rho_r:V_r\times U_{r^{-1}}\xrightarrow{\sim}
        V_r\times U_{r^{-1}}
\]
such that
\begin{equation}\label{eq:straightened-chart}
 \widetilde\chi_r(v,u,t):=u v n_rtU
   =\chi_r\bigl(\rho_r(v,u),t\bigr).
\end{equation}
It fixes the zero section \(\{1\}\times U_{r^{-1}}\times T\).
On the normal quotient along that section, its derivative preserves
the filtration
\[
 h_r(\delta)=\operatorname{ht}(-r^{-1}\delta),
 \qquad \delta\in r\Phi^-,
\]
and induces the identity on the associated graded; in particular its
determinant is \(1\).

For a cover \(w=w_0a\lessdot v=w_0b\), let \(y_q\) be the negative
root coordinate in \(V_w\) selected by deletion of the \(q\)-th
letter of \(\mathbf a\), and write the other normal coordinates as
\(\widehat{\mathbf y}\).  After restricting the straightened charts
to \(\mathcal W_{v,w}\), put
\[
 \mathscr O_w=\mathcal W_{v,w}\cap\mathcal O_w,
 \qquad
 \mathscr O_v=\mathcal W_{v,w}\cap\mathcal O_v.
\]
After shrinking these opens around their intersections with
\(\overline Y_{v,w}\), one has
scheme-theoretically
\begin{equation}\label{eq:straightened-pair}
 \overline Y_{v,w}\cap\mathscr O_w
   =\{\widehat{\mathbf y}=0\},\qquad
 \overline Y_{v,w}\cap\mathscr O_v
   =\{\mathbf y'=0\},
\end{equation}
where \(\mathbf y'\) denotes every target normal coordinate.  Thus
\(z=y_q\) trivializes \(L\), and the surviving source normal
coordinates and the target normal coordinates are frames of the
actual bundle \(E\) in \Cref{lem:two-stratum}.  We abbreviate
\(\widetilde\chi_a=\widetilde\chi_{w_a}\) and
\(\widetilde\chi_b=\widetilde\chi_{w_b}\).
\end{lemma}

\begin{proof}
Put \(\Psi_r=r\Phi^-\).  It is closed under root addition and contains
no opposite pair; its negative and positive parts are the root sets of
\(V_r\) and \(U_{r^{-1}}\), respectively.  The multiplication maps for
both block orders are therefore isomorphisms onto the same unipotent
group \(n_rU^-n_r^{-1}\).  Reorder \(uv\) into \(v'u'\).  Successive
Chevalley commutations give mutually inverse polynomial changes
\((v,u)\leftrightarrow(v',u')\), proving
\eqref{eq:straightened-chart}.  When \(v=1\), no commutator is created,
so the zero section is fixed.  For \(\delta\in\Psi_r\), the function
\[
 h_r(\delta)=\operatorname{ht}(-r^{-1}\delta)>0
\]
is additive on every root sum.  A commutator root
\(i\delta+j\epsilon\) has height
\(ih_r(\delta)+jh_r(\epsilon)\), strictly greater than either input.
Thus \Cref{lem:unipotent-triangular} applies even though
\(\Psi_r\) has mixed signs: every original normal coordinate retains
coefficient \(1\), while all new terms have strictly higher
\(h_r\)-degree.  This proves the filtration and determinant assertions.

Let \(\gamma_q\) be the actual negative root of \(V_w\) represented by
\(y_q\).  Strong exchange and the pinned rank-one chart show that
\[
 z\longmapsto x_{\gamma_q}(z)n_wB
\]
has value in \(C_w\) at \(z=0\) and in \(C_v\) for \(z\ne0\), with
multiplicity one.  Since every \(u\in U_{w^{-1}}\) lies in \(B\),
placing the tangent factor on the left keeps
\(u x_{\gamma_q}(z)n_wB\) in \(C_w\cup C_v\) for every \(u\) and
\(z\).  Hence the closed coordinate subscheme
\[
 S_w=\{u x_{\gamma_q}(z)n_wtU:
       u\in U_{w^{-1}},\ z\in\mathbb A^1,\ t\in T\}
     =\{\widehat{\mathbf y}=0\}
\]
lies in \(\overline Y_{v,w}\cap\mathscr O_w\).  It is smooth,
reduced, and irreducible, and
\[
\dim S_w=\ell(w)+1+\operatorname{rk}T
 =\ell(v)+\operatorname{rk}T
 =\dim\overline Y_{v,w}.
\]
The dense open stratum
\(X_v\cong\mathbb A^{\ell(v)}\times T\) is irreducible, hence
\(\overline Y_{v,w}\) and every nonempty open intersection used here
are irreducible.
The target \(\overline Y_{v,w}\cap\mathscr O_w\) of the displayed
inclusion is a nonempty open subscheme of the smooth irreducible scheme
\(\overline Y_{v,w}\); hence it is reduced and irreducible of the same
dimension.  The displayed inclusion is closed;
a proper closed subset of an irreducible finite-type scheme has smaller
dimension.  The two reduced schemes are therefore equal.  This proves
the first equality in \eqref{eq:straightened-pair} without identifying
Bott--Samelson parameters with tubular coordinates.

In the target chart, \(\{\mathbf y'=0\}=X_v\) is closed and dense in
\(\overline Y_{v,w}\cap\mathscr O_v\).  The same reducedness and
irreducibility argument proves the second equality.  Finally, the first
equality identifies \(z=0\) with
\(X_w\subset\overline Y_{v,w}\), with multiplicity one by the
rank-one calculation.  The remaining coordinates are consequently
the conormal parameters of the regular immersion and give the two
asserted frames of \(E\), compatibly with
\eqref{eq:normal-exact-sequence}.
\end{proof}

\begin{lemma}\label{lem:clutching}
Let \(V\) be a trivialized oriented vector bundle of rank \(p-1\) over
\(\mathbb A^d\times T\).  Suppose that, over the overlap
\(z\in\Gm\) of two purity charts for an additional oriented normal line,
the induced transition on \(V\times T\) is
\[
  (v,t)\longmapsto(A(z,\mathbf y)v,\lambda(z)t),
\]
where \(A\) is triangular in the affine coordinates,
\(\det A(z,\mathbf y)=z^m\), and \(A(1,\mathbf y)\) has determinant \(1\).
Use \(z=1\) to identify the two trivializations at the pointed
section.  The component of the cellular connecting morphism obtained
from the quotient triple and these clutching data is
\[
  \Dop_{\lambda,m}^{(p)}:
  \KMW1\otimes\KMW{p-1}\otimes\HT
  \longrightarrow\KMW{p-1}\otimes\HT.
\]
\end{lemma}

\begin{proof}
We first identify the connecting morphism before computing it.
Apply the quotient-triple construction
\eqref{eq:quotient-triple} to the two purity charts and excise their
complements.  Contracting their affine base coordinates reduces the
triple to the standard divisor pair
\[
 \{0\}\times T\subset\mathbb A^1\times T,
 \qquad
 (\mathbb A^1-\{0\})\times T=\Gm\times T,
\]
externally smashed with the Thom space of \(V\).  Use the first chart
over the divisor and the second chart over its complement.  In the
mapping-cone model of this quotient triple, the attaching map is the
transition map on the overlap, based by its value at \(z=1\).  More
explicitly, after applying reduced \(\Aone\)-homology, its boundary is
the composite
\begin{equation}\label{eq:clutching-reduced-map}
\begin{tikzcd}[column sep=large]
 \widetilde{\mathbf H}^{\Aone}_*
 \bigl(\Th(V)\wedge(\Gm\times T)_+\bigr)
 \arrow[r,"(\Th h)_*-(\Th h_1)_*"]
 &
 \widetilde{\mathbf H}^{\Aone}_*
 \bigl(\Th(V)\wedge T_+\bigr),
\end{tikzcd}
\end{equation}
where
\[
 h(z,v,t)=(A(z,\mathbf y)v,\lambda(z)t),
 \qquad
 h_1(z,v,t)=(A(1,\mathbf y)v,t).
\]
Indeed, the ordinary mapping cone of a map of pointed simplicial
presheaves has boundary equal to its attaching map.  Changing from
the first trivialization to the second replaces that attaching map by
\(\Th h\), while the constant identification at the pointed section
is \(\Th h_1\); hence their difference after applying the additive
homology functor.  At \(z=1\) the two maps agree, so
\eqref{eq:clutching-reduced-map} vanishes on
\(\{1\}\times T\) and factors uniquely through
\[
 \ZA[\Gm]\longrightarrow\ZA(\Gm)=\KMW1.
\]
This proves, at the level of the actual quotient triple, that the
connecting morphism is the reduced clutching class; it does not infer
the boundary from a comparison of its realizations.

It remains to compute that reduced class.  On symbols,
\eqref{eq:clutching-reduced-map} is represented by
\[
  z\longmapsto
  \bigl(\Th(A(z,\mathbf y))\wedge[\lambda(z)t]\bigr)
  -\bigl(\Th(A(1,\mathbf y))\wedge[t]\bigr).
\]
An automorphism of an oriented vector bundle acts on its Thom homology
by the class of its determinant.  The strictly triangular part has
determinant \(1\), and the normalized variable part
\(A(z,\mathbf y)A(1,\mathbf y)^{-1}\) contributes
\(\ang{z^m}\).  Thus the reduced class acts as
\[
  \ang{z^m}x\otimes[\lambda(z)t]-x\otimes[t].
\]
Applying reduced
\(\Aone\)-homology and
\(\widetilde{\mathbf H}^{\Aone}_0(\Gm)=\KMW1\) gives exactly
\Cref{eq:D-symbol}.  The two Thom maps and the quotient triple were
formed in the pointed \(\Aone\)-homotopy category; subtraction occurs
only after applying reduced \(\Aone\)-homology.  Hence the formula is
a morphism of strictly \(\Aone\)-invariant sheaves, not merely a rule
on field-valued symbols.
\end{proof}

\begin{lemma}
\label{lem:opposition-bookkeeping}
In the coordinates of \eqref{eq:opposition-bott-samelson}, let
\(z\) be the \(q\)-th coordinate and put
\[
 \beta_r^{\mathrm{tail}}=
 s_{i_{q+1}}\cdots s_{i_{r-1}}(\alpha_{i_r})
 \qquad(q<r\leq p).
\]
Then
\(\{\beta_r^{\mathrm{tail}}:q<r\leq p\}=\Pi_{a,q}\).
On \(z\ne0\), eliminating the \(q\)-th simple reflection translates
the right torus coordinate by
\(\lambda_{a,q}(z)=t_{a,q}^{-1}(\alpha_{i_q}^\vee)(z)\).
The prefix survivor axes have diagonal weight \(1\), while the tail
axis labelled by \(\beta\in\Pi_{a,q}\) has diagonal weight
\(z^{\langle\beta,\alpha_{i_q}^\vee\rangle}\).  Every other
root-coordinate term created in this elimination is a unipotent
triangular shear for a transported root-height filtration.
\end{lemma}

\begin{proof}
The pinned rank-one identity appropriate to the opposition product is
\begin{equation}\label{eq:positive-rank-one}
 x_{\alpha_i}(z)n_i
 =x_{-\alpha_i}(z^{-1})x_{\alpha_i}(-z)
  \alpha_i^\vee(z).
\end{equation}
For the standard pinned matrices, both sides are
\(\left(\begin{smallmatrix}z&-1\\1&0\end{smallmatrix}\right)\).
Thus \eqref{eq:positive-rank-one} removes the \(q\)-th Tits factor,
supplies the additional target tangent coordinate \(z^{-1}\), and
places \(\alpha_{i_q}^\vee(z)\) immediately to the left of the tail.

Move a current cocharacter \(\mu(z)\) through one tail factor by
\[
 \mu(z)x_{\alpha_j}(y)
 =x_{\alpha_j}(z^{\langle\alpha_j,\mu\rangle}y)\mu(z),
 \qquad
 \mu(z)n_j=n_j(s_j^{-1}\mu)(z).
\]
Before the \(r\)-th tail factor the current cocharacter is
\((s_{i_{q+1}}\cdots s_{i_{r-1}})^{-1}\alpha_{i_q}^\vee\).
Its exponent on that factor is therefore
\[
 \left\langle\alpha_{i_r},
 (s_{i_{q+1}}\cdots s_{i_{r-1}})^{-1}
 \alpha_{i_q}^\vee\right\rangle
 =\langle\beta_r^{\mathrm{tail}},\alpha_{i_q}^\vee\rangle.
\]
After the last factor the cocharacter is
\(t_{a,q}^{-1}\alpha_{i_q}^\vee=\lambda_{a,q}\).  The standard
inversion-set formula for the reduced tail gives the asserted equality
with \(\Pi_{a,q}\).  No prefix factor is crossed.

We make the role of the factor \(x_{\alpha_{i_q}}(-z)\) in
\eqref{eq:positive-rank-one} explicit, since it is the one term of the
rank-one identity that does not fit the diagonal pattern.  It is a
root-subgroup element of the \emph{deleted} simple root, evaluated at
\(-z\): under the subsequent reorderings it either merges with a
surviving coordinate of the same root line, shifting that coordinate by
a function of \(z\) alone, or produces commutator terms in strictly
higher roots.  In both cases it changes coordinates only by terms that
are constant in the surviving normal variables or strictly triangular
for the transported height; it therefore shifts tangent or deleted
coordinates without contributing to the derivative on the surviving
normal coordinates, and in particular it affects neither the diagonal
weights nor the determinant computed below.  The factor
\(x_{-\alpha_{i_q}}(z^{-1})\) likewise supplies only the new target
tangent coordinate displayed in
\eqref{eq:full-tubular-overlap}.

Finally, reorder the
two unipotent factors in \eqref{eq:positive-rank-one} and the survivor
factors into the target order.  By
\Cref{lem:unipotent-triangular}, each such reordering has unit diagonal
for the transported height and creates no torus term.  This proves all
claims.
\end{proof}

\begin{lemma}
\label{lem:weighted-normal}
Let a product of pinned root-subgroup coordinates be transformed by a
finite sequence of the following operations:
\begin{enumerate}
\item adjacent Chevalley reorderings inside a closed root set with no
opposite pair;
\item conjugations by root-subgroup elements belonging to the same
unipotent group;
\item a diagonal rescaling
\(y_\rho\mapsto z^{m_\rho}y_\rho\), with \(z\in\Gm\).
\end{enumerate}
Suppose that the coordinate locus \(\mathbf y=0\) is preserved, and
regard the remaining coordinates \(\mathbf x\) as tangent parameters.
After the source and target normal axes are identified by their root
labels, the derivative on the normal quotient has the form
\[
 D_{\mathbf y}F(z,0,\mathbf x)
 =U_1(z,\mathbf x)
   \operatorname{diag}(z^{m_\rho})_\rho
   U_2(z,\mathbf x)P,
\]
where \(P\) is the constant invertible root-axis frame comparison and
\(U_1,U_2\) are unipotent triangular matrices.  In particular,
\[
 \det D_{\mathbf y}F(z,0,\mathbf x)
 =\det(P)\,z^{\sum_\rho m_\rho}.
\]
No coefficient occurring only in a higher-root commutator changes this
determinant.
\end{lemma}

\begin{proof}
Use the positive additive height of
\Cref{lem:unipotent-triangular}.  For one adjacent reordering, the two
original coordinates retain coefficient \(1\), while every commutator
coordinate has strictly larger height.  Differentiate after imposing
\(\mathbf y=0\).  Terms independent of \(\mathbf y\) affect only the
tangent map, and terms linear in a normal variable can enter only a
strictly higher normal root coordinate.  Thus the induced normal
 matrix is a constant root-axis frame comparison followed by a
 unipotent triangular matrix, even when its off-diagonal entries depend on
\(\mathbf x\).  The same argument applies to a conjugation.  Composing
the elementary matrices and collecting the diagonal rescaling gives
the displayed factorization.  The possible Chevalley constants occur
only in the strict triangular entries, so both unipotent factors have
determinant \(1\) over the ground ring; no division by those constants
is used.
\end{proof}

\begin{lemma}\label{lem:cover-neighborhood}
After replacing \(\mathcal W_{w_b,w_a}\) by its intersection with
\(\mathcal O_{w_a}\cup\mathcal O_{w_b}\), the restrictions of the two
straightened tubular charts
\(\widetilde\chi_a,\widetilde\chi_b\) cover a Zariski neighbourhood of
\(Y_a\cup Y_b\).  Let \(z\) be the \(q\)-th normal coordinate in
\(\widetilde\chi_a\).  On the part of their overlap over
\(\overline Y_{w_b,w_a}\), where \(z\ne0\), write the remaining normal
coordinates as \(\widehat{\mathbf y}\) and the affine coordinates along
\(Y_a\) as \(\mathbf x\).  On an ambient neighbourhood of that part,
\begin{equation}\label{eq:full-tubular-overlap}
 \widetilde\chi_b^{-1}\widetilde\chi_a
 (z,\widehat{\mathbf y},\mathbf x,t)
 =\bigl(
   \mathbf y'=F(z,\widehat{\mathbf y},\mathbf x),
   \mathbf x'=(z^{-1},P(z,\widehat{\mathbf y},\mathbf x)),
   \lambda_{a,q}(z)t
  \bigr),
\end{equation}
where \(F\) and \(P\) are root-reordering morphisms whose coordinate
functions lie in
\[
 k[z,z^{-1},\widehat{\mathbf y},\mathbf x].
\]
Moreover, \(F(z,0,\mathbf x)=0\) by
\eqref{eq:straightened-pair}.  The displayed \(z^{-1}\) is the
additional tangent coordinate of \(Y_b\).  On the linearization of the
actual bundle \(E\), up to transformations strictly triangular in root
height, the surviving normal coordinates transform as
\begin{equation}\label{eq:local-overlap}
  (\widehat{\mathbf y},t)\longmapsto
  \left(
  (z^{\langle\beta,\alpha_{i_q}^{\vee}\rangle}y_\beta)
  _{\beta\in\Pi_{a,q}},
  \lambda_{a,q}(z)t\right).
\end{equation}
Its determinant on the tail affine coordinates is
\(z^{\sigma_{a,q}}\).  The remaining constant change of frame has degree
\(\vartheta_{a,q}\).
More intrinsically, the transition matrix of \(E\) is the derivative
along its survivor-zero locus, and it factors as
\begin{equation}\label{eq:normal-derivative-factorization}
 A(z,\mathbf x):=
 D_{\widehat{\mathbf y}}F(z,0,\mathbf x)
 =U_1(z,\mathbf x)\,
 \operatorname{diag}\!\left(
 I_{q-1},
 \bigl(z^{\langle\beta,\alpha_{i_q}^{\vee}\rangle}\bigr)
 _{\beta\in\Pi_{a,q}}
 \right)
 U_2(z,\mathbf x)\,A_0,
\end{equation}
where \(U_1,U_2\) are unipotent for the appropriate transported root
height filtrations, and \(A_0\) is the constant deleted-word frame
comparison.  Thus, in the deleted-word target frame,
\begin{equation}\label{eq:reduced-E-transition}
 A^{\mathrm{red}}(z,\mathbf x)=A(z,\mathbf x)A_0^{-1},
 \qquad
 \det A^{\mathrm{red}}(z,\mathbf x)=z^{\sigma_{a,q}},
 \qquad
 \det A^{\mathrm{red}}(1,\mathbf x)=1.
\end{equation}
After contracting the affine base
coordinates, the induced base map is exactly
\((z,t)\mapsto\lambda_{a,q}(z)t\).
\end{lemma}

\begin{proof}
By \Cref{lem:two-stratum}, the cover is represented by a unique
multiplicity-one Cartier divisor, and
\Cref{lem:left-straightening} makes that actual two-stratum scheme the
coordinate locus \(\widehat{\mathbf y}=0\).  The deleted coordinate
\(z=y_q\) is therefore a uniformizer for \(L\); its possible constant
rescaling is part of the frame comparison \(A_0\).  Use this coordinate
in the full multiplication formula for \(\widetilde\chi_a\).  Apply
left translation by \(n_{w_0}^{-1}\).  This is an automorphism of
\(X\), commutes with the right \(T\)-action, and carries the actual
normal factors to the positive Bott--Samelson product
\eqref{eq:opposition-bott-samelson}.  In particular the \(q\)-th
factor is literally \(x_{\alpha_{i_q}}(z)n_{i_q}\); there is no
identification of the actual root \(w_0\beta_q\) with the generally
different root \(-\beta_q\).

Now apply \Cref{lem:opposition-bookkeeping}.  The identity
\eqref{eq:positive-rank-one} supplies the new target tangent coordinate,
removes the \(q\)-th reflection, and moves the cocharacter through the
tail.  It gives exactly the right translation
\(\lambda_{a,q}(z)\), the diagonal tail weights in
\eqref{eq:local-overlap}, and only unit-diagonal unipotent reordering
terms.  Because the deleted Tits word is a reduced word for \(b\), its
product is the fixed \(n_b\); multiplying again by \(n_{w_0}\) returns
exactly the target lift \(\dot w_b\) in the opposition gauge.

The tangent factor in a straightened chart remains on the left.
Reorder all remaining source normal factors, and then put the result
uniquely in the target straightened order
\(U_{w_b^{-1}}V_{w_b}n_{w_b}T\).  These Chevalley identities define
the polynomials \(F,P\) in \eqref{eq:full-tubular-overlap}; their only
denominators are powers of \(z\), so they are regular on \(z\ne0\).
These Chevalley reordering formulas are identities over all affine and
torus coordinates.  Moreover,
\eqref{eq:straightened-pair} says scheme-theoretically that both charts
send the actual two-stratum scheme to their indicated coordinate
loci.  Hence \(F(z,0,\mathbf x)=0\).

The tail roots are precisely
\(\Pi_{a,q}=\Phi^+\cap t_{a,q}\Phi^-\), and the diagonal entry on the
\(\beta\)-root line is
\(z^{\langle\beta,\alpha_{i_q}^{\vee}\rangle}\).  Their product is
\[
  z^{\sum_{\beta\in\Pi_{a,q}}
       \langle\beta,\alpha_{i_q}^{\vee}\rangle}
  =z^{\sigma_{a,q}}.
\]
The prefix block is likewise a unipotent triangular shear; it need not
be literally the identity, but its determinant is \(1\).
Apply \Cref{lem:weighted-normal} to the normal quotient.  It shows
directly that all root-subgroup reorderings assemble into the two
unipotent blocks in
\eqref{eq:normal-derivative-factorization}, while the only
\(z\)-dependent diagonal block is the one displayed in
\eqref{eq:local-overlap}.  The constant signed permutation is part of
the deleted-word frame comparison
\(\vartheta_{a,q}\), and there is no further \(z\)-dependent
determinant, including in types \(B_2\) and \(G_2\).
This is precisely
\eqref{eq:normal-derivative-factorization}: the two unipotent blocks
have determinant one, the middle torus block has the displayed weights,
and every remaining factor is constant.

The preceding rank-one calculation is also the transition of the
global bundle \(E\).  Indeed, the source and target straightening maps
of \Cref{lem:left-straightening} induce unipotent changes for the
transported heights \(h_{w_a}\) and \(h_{w_b}\)
on the corresponding normal quotients, both of determinant \(1\).
Conjugating by these changes only modifies \(U_1\) and \(U_2\) in
\eqref{eq:normal-derivative-factorization}; it changes neither the
diagonal powers of \(z\) nor their product
\(z^{\sigma_{a,q}}\).  Thus the factorization concerns the actual
normal bundle of the global purity model, not an auxiliary transverse
slice.

Finally compare the deleted-word chart with the fixed chart for \(b\);
its constant Jacobian is, by definition, the class
\(\vartheta_{a,q}\).  The braid-compatible Tits representatives fixed
above make the two horizontal lifts equal.  With arbitrary lifts one
would obtain, in addition, a constant right translation on \(T\);
\Cref{prop:choices} explains why this only conjugates the complex.
The maps \(P\) and the non-normal part of \(F\) take values in affine
root groups.  More explicitly, multiplying every affine target
coordinate, including the displayed \(z^{-1}\), by a homotopy
parameter \(s\) gives a regular \(\Aone\)-homotopy over
\(\Gm\times T\) from the affine base map to its zero section.
Similarly, a strictly triangular matrix \(I+N\) is joined to the
identity by \(I+sN\), which stays invertible with determinant \(1\).
We verify that these are homotopies of the pointed pairs actually
used.  The target chart \(\widetilde\chi_b\) identifies its purity
quotient with the Thom space of the trivialized bundle \(E\) over the
affine-times-torus base
\(\mathbb A^{\ell(b)}\times T\); in these coordinates the two
homotopies above are algebraic families, over
\(\mathbb A^1\times\Gm\times T\), of bundle maps of \(E\).  A bundle
map sends fibers to fibers and the complement of the zero section to
the complement of the zero section (the matrices \(I+sN\) are
invertible for every \(s\)), so each stage of the homotopy descends to
the pointed Thom quotient
\(\Th(E)=E/(E-0)\); pointedness is preserved because the zero section
maps to the zero section at every \(s\).  Since the coordinate
functions of the homotopy lie in
\(k[s,z,z^{-1},\widehat{\mathbf y},\mathbf x]\), the family stays
inside the fixed coordinate opens: no point leaves the chart domain,
because the chart is an isomorphism onto its image and the formulas
above are compositions of chart maps.  The torus coordinate in
\eqref{eq:full-tubular-overlap} is not contracted.  Thus the induced
map on the base of the target Thom space is
\((z,t)\mapsto\lambda_{a,q}(z)t\), as asserted.
\end{proof}

\begin{lemma}\label{lem:sigma-height}
Put \(\gamma_{a,q}=t_{a,q}^{-1}(\alpha_{i_q})\).  For a reduced
deletion this is a positive root and
\[
  \sigma_{a,q}=1-\operatorname{ht}(\gamma_{a,q}^{\vee}).
\]
In particular,
\(\sigma_{a,q}\equiv
\operatorname{ht}(\gamma_{a,q}^{\vee})+1\pmod2\).
\end{lemma}

\begin{proof}
Let \(\rho\) be the half-sum of the positive roots.  The inversion-set
identity gives
\[
  \rho-t_{a,q}\rho
  =\sum_{\beta\in\Phi^+\cap t_{a,q}\Phi^-}\beta.
\]
Pairing with \(\alpha_{i_q}^{\vee}\), and using
\(\langle\rho,\alpha_{i_q}^{\vee}\rangle=1\), gives
\[
  \sigma_{a,q}
  =1-\langle\rho,t_{a,q}^{-1}\alpha_{i_q}^{\vee}\rangle
  =1-\operatorname{ht}(\gamma_{a,q}^{\vee}).
\]
\end{proof}

\begin{theorem}\label{thm:all-degree}
Under \eqref{eq:X-chains-a}, a component of \(\partial_p\) is zero unless
it corresponds to a reduced deletion.  If several deletion positions
in \(\mathbf a\) produce the same target \(b\), the matrix entry from
the \(a\)-summand to the \(b\)-summand is the sum of the following
contributions over those positions.  For deletion at position \(q\),
\begin{equation}\label{eq:main-boundary}
  \partial_{a,b}=(-1)^{q-1}
  \vartheta_{a,q}\,
  \Dop_{\lambda_{a,q},\sigma_{a,q}}^{(p)}
  \circ(\operatorname{cut}_{a,q}\otimes\operatorname{id}_{\HT}).
\end{equation}
Here \(\vartheta_{a,q}\) acts by scalar multiplication on the target
\(\KMW{p-1}\)-factor.
In scalar Thom bases this reads
\begin{equation}\label{eq:main-boundary-symbol}
  (u)\otimes x\otimes[t]\longmapsto
  (-1)^{q-1}\vartheta_{a,q}
  \left(
    \ang{u^{\sigma_{a,q}}}x\otimes
    [\lambda_{a,q}(u)t]-x\otimes[t]
  \right).
\end{equation}
The sum of these components is the cellular differential of \(G/U\)
and, through \Cref{thm:G-X}, the Morel--Sawant differential of \(G\).
\end{theorem}

\begin{proof}
The exchange condition lists exactly the codimension-one Bruhat strata
in the boundary.  \Cref{lem:two-stratum} isolates the chosen cover as a
global smooth two-stratum purity model and identifies the cellular
component with the connecting map of
\eqref{eq:two-stratum-thom-cofiber}.  It also gives multiplicity one.
On the overlap of the left-\(B\)-adapted source and target Bruhat
tubular charts, \Cref{lem:cover-neighborhood} identifies the
clutching automorphism of the actual bundle \(E\):
the torus is translated by \(\lambda_{a,q}(u)\), its determinant on the
remaining normal directions is \(u^{\sigma_{a,q}}\), and the constant
frame comparison is \(\vartheta_{a,q}\).  Because
\eqref{eq:full-tubular-overlap} holds over every affine and torus base
coordinate, this is a calculation of the Thom map over the whole target
stratum, rather than at one generic point.

After replacing the transition by \(A^{\mathrm{red}}\) in
\eqref{eq:reduced-E-transition}, \Cref{lem:clutching} identifies the
connecting morphism with the universal operator in the deleted-word
target frame.  Postcomposing with
the change to the fixed target Thom frame multiplies both endpoints by
\(\vartheta_{a,q}\).  Separating the \(q\)-th Thom factor gives
\(\operatorname{cut}_{a,q}\).  The ordered cofiber calculation of
\Cref{lem:face-sign} gives the ordinary chain sign
\((-1)^{q-1}\), independently of the Milnor--Witt symmetry already
contained in \(\operatorname{cut}_{a,q}\).  This proves
\eqref{eq:main-boundary-symbol}.  The commutative localization square
in the proof of \Cref{lem:incidence-localization} identifies this local
map with the selected component of the global cellular connecting
morphism.
\end{proof}

\begin{corollary}\label{cor:square-zero}
The operators in \eqref{eq:main-boundary} satisfy
\(\partial_{p-1}\partial_p=0\) as sheaf morphisms.
\end{corollary}

\begin{proof}
They are the connecting morphisms of the cellular exact couple.  Hence
\(\partial^2=0\) before any realization or augmentation.
\end{proof}

\begin{remark}
\Cref{cor:square-zero} is a formal consequence of the exact couple and
is available only \emph{after} \Cref{thm:all-degree} has identified
the displayed operators with the genuine cellular connecting
morphisms.  It is therefore not an independent verification of the
closed formula, and it is not used as one anywhere in this paper.
The independent checks are of a different kind: the direct symbolic
\(A_2\) computation of \Cref{prop:SL3-square-zero}, the direct purity
computation of \Cref{prop:SL3-top-check}, the realized comparisons
with \(S^1\), \(SO(3)\), and \(SO(4)\)
(\Cref{sec:examples,sec:G2}), and the realized rank-two diamond
checks, which confirm square zero for the displayed matrices by direct
computation in every rank-two type (\Cref{rem:rank-two-diamonds}).
\end{remark}

\begin{proposition}\label{prop:choices}
Changing any of the choices used in this paper---normal words,
root-product orderings, pinned root-parameter scalings, or horizontal
Tits lifts---conjugates the displayed complex by a diagonal chain
isomorphism.  On each summand the diagonal factor is a
Grothendieck--Witt unit times a right translation in \(\HT\).  This
statement concerns these root-product choices; it does not assert the
same form for an arbitrary non-equivariant bundle trivialization.
\end{proposition}

\begin{proof}
Let \(\mathfrak c\) and \(\mathfrak c'\) be two complete systems of
choices.  On the summand indexed by \(a\), compare the two \emph{actual}
normal-bundle trivializations and the two actual horizontal lifts.  The
first comparison has a constant determinant
\(d_a\in k^\times\), and the second is right translation by a constant
\(\tau_a\in T(k)\).  They define the diagonal isomorphism
\[
 S_a=\ang{d_a}R_{\tau_a}:
 \KMW{\ell(a)}\otimes\HT\,e_a
 \xrightarrow{\sim}
 \KMW{\ell(a)}\otimes\HT\,e_a.
\]
Functoriality of homotopy purity for a change of Thom trivialization and
for right translation gives, for every cover \(b\lessdot a\),
\[
 S_b\,\partial^{\mathfrak c}_{a,b}
 =\partial^{\mathfrak c'}_{a,b}\,S_a.
\]
Taking the direct sum of the \(S_a\) proves the proposition.  In
particular, coherence around a loop of braid moves is automatic: the
product is the Jacobian of the resulting actual coordinate
automorphism, not a sign assigned separately to a braid path.
\end{proof}

\subsection{Low-degree checks}

\begin{proposition}\label{prop:MS-check}
Under the ordered Thom bases of \cite[Section~4]{MorelSawant2023},
\eqref{eq:main-boundary} agrees term by term with
\cite[Proposition~4.28]{MorelSawant2023} in degree two.  After passage
to their cycle bases and torus augmentation, it gives
\cite[Theorems~4.30 and~4.39]{MorelSawant2023}, for both orders of the
deleted positions in the chosen expression of \(w_0\).
\end{proposition}

\begin{proof}
Write
\[
 (r)_i=[\alpha_i^\vee(r)]-[1]\in\HT,
 \qquad n=n_{ji}=-\langle\alpha_j,\alpha_i^\vee\rangle.
\]
Morel--Sawant's integer \(\lambda_i^{\mathrm{MS}}\), which records a
position in a reduced expression of \(w_0\), is unrelated to our
cocharacter \(\lambda_{a,q}\).  For \(a=s_is_j\), deletion of the first
letter has
\[
 \lambda_{a,1}=s_j(\alpha_i^\vee)
 =\alpha_i^\vee+n\alpha_j^\vee,\qquad \sigma_{a,1}=-n,
\]
whereas deletion of the second has
\(\lambda_{a,2}=\alpha_j^\vee\) and \(\sigma_{a,2}=0\).
Choose the compatible Morel--Sawant frames, so both constant frame
comparisons are \(\ang1\).  Retaining their notation for the
right-translated reduced torus classes, the component to the \(j\)-summand is
\begin{align}
 &\ang{u^{-n}}(v)\otimes
    [\widetilde w_j s_j(\alpha_i^\vee(u))]
       -(v)\otimes[\widetilde w_j]\notag\\
 &\quad=\ang{u^{-n}}(v)\otimes\widetilde w_j s_j(u)_i
   +\bigl(\ang{u^{-n}}-1\bigr)(v)
       \otimes[\widetilde w_j].\label{eq:MS-termwise}
\end{align}
Their unit \(\rho_{ij}\) satisfies
\[
 \ang{\rho_{ij}}=\ang{u^{-n}},\qquad
 \eta(\rho_{ij})=\ang{u^{-n}}-1.
\]
Thus the two terms in \eqref{eq:MS-termwise} are exactly the first two
terms of \cite[Proposition~4.28]{MorelSawant2023}.  The other deletion
is
\[
 -(u)\otimes\widetilde w_i(v)_j,
\]
its last term.  Taking a one-letter word gives their rank-one coroot
differential.

For the augmented comparison, recall their cycles
\[
 \begin{aligned}
 \phi_j(v,r)&=(v)\otimes\widetilde w_j(r)_j
       -\eta(v)(r)\otimes[\widetilde w_j],\\
 \delta_{ji}(v,u)&=(v)\otimes\widetilde w_j(u)_i
       -(u)\otimes\widetilde w_i(v)_j.
 \end{aligned}
\]
If \(n=2\), then \(\ang{u^{-2}}=1\) and
\[
 s_j(u)_i=(u)_i+(u^2)_j+(u)_i(u^2)_j.
\]
Put \(h=1+\ang{-1}\) and
\((-1)_i=[\alpha_i^\vee(-1)]-[1]\).  The identities
\((u^2)=h(u)\), \(\eta h=0\), and
\((u)_i(u^2)_j=(-1)_i(u^2)_j\) give, before augmentation,
\[
 \partial_2(u,v)=\delta_{ji}(v,u)
   +\phi_j(v,u^2)(1+(-1)_i).
\]
After \(-\otimes_{\HT}\Z\), this becomes
\[
 \ol\partial_{ij}((u)(v))
 =\ol\delta_{ji}(v,u)+h\,\ol\phi_j(v,u),
\]
which is \cite[Theorem~4.30(b)]{MorelSawant2023}.  For odd \(n\),
\(\ang{u^{-n}}=\ang u\), and the same expansion gives their
\((n_{ji})_\eps\ol\phi_j(v,u)\) correction.

If \(\lambda_i^{\mathrm{MS}}<\lambda_j^{\mathrm{MS}}\), change the
reduced expression as in Morel--Sawant.  With
\(\varsigma_r=(\breve w_r)^{-1}\widetilde w_r\), their comparison map is
\[
 (x)\otimes[\widetilde w_rt]
 \longmapsto
 \ang{\tau_r}(x)\otimes[\breve w_r\varsigma_rt].
\]
This is exactly the diagonal choice-change map
\(\ang{\tau_r}R_{\varsigma_r}\) of \Cref{prop:choices}.  After
augmentation it fixes \(\delta_{ji}\) and multiplies \(\phi_j\) by
\(\ang{\tau_j}\).  Hence the odd correction acquires that factor,
while the even one does not because \(h\ang{\tau_j}=h\).  These are
the clauses of \cite[Theorem~4.39]{MorelSawant2023}.
\end{proof}

\begin{example}
\label{ex:sp4}
Let \(\alpha_1\) be short and \(\alpha_2\) long, so
\[
 -\langle\alpha_2,\alpha_1^\vee\rangle=2,
 \qquad
 -\langle\alpha_1,\alpha_2^\vee\rangle=1.
\]
Use the normal words \(s_1,s_2,s_1s_2,s_2s_1\).  Every degree-two
deletion is already the chosen target word, so all four frame factors
are \(\ang1\).  The cover data are
\[
\begin{array}{ccccc}
\toprule
a&q&b&\lambda_{a,q}&\sigma_{a,q}\\
\midrule
s_1s_2&1&s_2&\alpha_1^\vee+2\alpha_2^\vee&-2\\
s_1s_2&2&s_1&\alpha_2^\vee&0\\
s_2s_1&1&s_1&\alpha_1^\vee+\alpha_2^\vee&-1\\
s_2s_1&2&s_2&\alpha_1^\vee&0\\
\bottomrule
\end{array}
\]
Consequently, with target rows \((s_1,s_2)\), source columns
\((s_1s_2,s_2s_1)\), and the indicated cuts understood,
\[
 \partial_2=
 \begin{pmatrix}
 -\Dop_{\alpha_2^\vee,0}^{(2)}\operatorname{cut}_{12,2}
 &\Dop_{\alpha_1^\vee+\alpha_2^\vee,-1}^{(2)}
       \operatorname{cut}_{21,1}\\[2pt]
 \Dop_{\alpha_1^\vee+2\alpha_2^\vee,-2}^{(2)}
       \operatorname{cut}_{12,1}
 &-\Dop_{\alpha_1^\vee,0}^{(2)}\operatorname{cut}_{21,2}
 \end{pmatrix}.
\]
The two nontrivial normal determinants are \(u^{-2}\) and \(u^{-1}\).
Thus the even Cartan edge has trivial variable Grothendieck--Witt class
but retains the doubled coroot in its torus translation, whereas the
odd edge has class \(\ang u\).  This is the \(n_{21}=2\) term checked
in \Cref{prop:MS-check}.

The normalized long-braid transition
\(1212\leftrightarrow2121\) has linear part
\[
 (y_1,y_2,y_3,y_4)\longmapsto(y_4,-y_3,y_2,y_1).
\]
Its determinant is \(-1\), so its Milnor--Witt degree is
\(\ang{-1}\), in agreement with
\cite[Lemma~4.2]{LiuLiu2026}.  The Chevalley coefficient \(2\) occurs
only in higher-order off-diagonal terms and changes neither this
Jacobian nor the diagonal product \(u^{-2}\).
\end{example}

\begin{corollary}\label{cor:type-A}
For \(SL_n\) with the standard pinning, every
\(\vartheta_{a,q}\) is the product of the degrees of braid moves from
\(\widehat{\mathbf a}_q\) to \(\mathbf b\).  Commutations contribute
\(\ang{-1}\), and three-term \(A_2\) braids contribute \(\ang1\).
Thus \eqref{eq:main-boundary} is an effective all-degree formula from
reduced words and the Cartan matrix.
\end{corollary}

\begin{proof}
Matsumoto's theorem connects the deleted word to the chosen target word
by commuting moves and three-term \(A_2\) braids.  In root-labelled
coordinates a commuting move interchanges two axes and has determinant
\(-1\).  For an \(A_2\) braid the three root axes are reversed, which
has determinant \(-1\), while the middle root vector changes from
\([e_{\alpha},e_{\beta}]\) to
\([e_{\beta},e_{\alpha}]\), contributing a second factor \(-1\).
Thus its determinant is \(+1\).  This agrees with
\Cref{prop:rank-two-degrees} and with the rank-two calculation
of \cite[Lemmas~4.2 and~4.5]{LiuLiu2026}.  Degrees multiply under
composition.  Their product is path independent because it is the
Jacobian determinant of the single polynomial automorphism
\(\Phi_{\mathbf b}^{-1}\Phi_{\widehat{\mathbf a}_q}\).  The reduced
word, deletion position, Cartan matrix, and a Matsumoto sequence
therefore give a finite algorithm for every entry.
\end{proof}

\begin{example}
\label{ex:SL4-frame}
In \(SL_4\), take
\(\mathbf a=(3,2,1)\) and delete \(q=2\).  The deleted word is
\((3,1)\), while choose \(\mathbf b=(1,3)\) for the same target.  The
single commuting move gives
\(\vartheta_{a,2}=\ang{-1}\).  Moreover
\[
 t_{a,2}=s_1,\qquad
 \lambda_{a,2}=s_1(\alpha_2^\vee)
   =\alpha_1^\vee+\alpha_2^\vee,
 \qquad \sigma_{a,2}=-1.
\]
Hence, in scalar cut bases, this cover sends
\[
 (u)\otimes x\otimes[t]\longmapsto
 -\ang{-1}\left(
   \ang{u^{-1}}x\otimes
   [ (\alpha_1^\vee+\alpha_2^\vee)(u)t]
   -x\otimes[t]
 \right).
\]
This is the first frame factor not visible in the normal-word choices of
the \(SL_3\) example below.
\end{example}

\section{Real realization and the maximal compact complex}
\label{sec:real}

Assume in this section that \(k=\R\).

\begin{remark}\label{rem:GR-connected}
For semisimple \emph{simply connected} \(G/\R\), the Lie group
\(G(\R)\) is automatically connected
\cite[Proposition~7.6]{PlatonovRapinchuk1994}; for split \(G\) this
also follows from Steinberg's generation of \(G(\R)\) by the connected
root subgroups \(x_\alpha(\R)\cong(\R,+)\)
\cite[Lemma~32 ff.]{Steinberg1968}.  The connectedness
hypothesis of \cite{PatraoSandoval2026} is therefore satisfied and is
not an additional assumption here.
\end{remark}

Let \(\theta\) be the Cartan
involution determined by the fixed pinning: on every pinned rank-one
subgroup it is the standard involution
\(g\mapsto(g^{\mathsf T})^{-1}\), and it restricts to inversion on
\(T\).  We take
\(K=G(\R)^\theta\), the resulting pinning-compatible maximal compact
subgroup, and let \(A=T(\R)^0\).  All identifications in this section
refer to this compatible choice: the identification
\(M=T[2](\R)\) below, the Iwasawa projection \(r\), and the pinned
representatives all depend on the fixed pinning-compatible
Cartan--Iwasawa data \((\theta,K,A)\).  Conjugating the Iwasawa data
transports
the comparison to any other maximal compact subgroup, but the pinned
representatives below then have to be conjugated as well.  Put
\[
  M:=K\cap T(\R)=Z_K(A)=T[2](\R).
\]
Inclusion of the sign representatives induces a canonical isomorphism
\[
 M\xrightarrow{\sim}T(\R)/T(\R)^0=\pi_0(T(\R));
\]
we use this isomorphism, rather than a literal equality of a subgroup
and a quotient, throughout.  For a simply connected split group,
\(M\) is generated by
\(c_i=\alpha_i^\vee(-1)\).  The pinned Tits representatives lie in
\(N_K(A)\), and their squares are the \(c_i\).
Patr\~ao--Sandoval's normalized root vector \(E_{\alpha_i}\) is
determined only up to sign.  We make the choice for which their
midpoint representative
\[
 s_i^{\mathrm{PS}}=\exp(F_i/2),\qquad
 F_i=E_{\alpha_i}+\theta E_{\alpha_i},
\]
is our \(n_i\).  Equivalently, if \(E_{\pm\alpha_i}\) now denote the
pinned Chevalley tangent vectors in the rank-one subgroup, this
convention reads
\(F_i=\pi(E_{-\alpha_i}-E_{\alpha_i})\).
Reversing one choice replaces \(n_i\) by
\(n_i^{-1}=n_ic_i\).  It therefore reverses the corresponding
cellular coordinate and right-translates the lifted cell label by
\(c_i\); on both complexes this is the same diagonal
sign-and-\(M\)-translation basis change.

Let
\[
  \operatorname{Re}_{\R}:\mathcal H_\bullet(\R)
  \longrightarrow\mathcal H_\bullet^{\mathrm{top}}
\]
denote real Betti realization of pointed motivic spaces
\cite[Section~3.3]{MorelVoevodsky1999}, cf.\
\cite{DuggerIsaksen2004}; on the filtered spaces used
here it commutes with the formation of quotients and Thom spaces of
real vector bundles, since these are realized to their topological
counterparts, and in particular it carries the purity cofiber
sequences of \Cref{sec:prelim} to the topological cofiber sequences of
the realized filtration.  In this
section it is always applied to the filtered spaces
\(\Omega_\bullet^X\) and their purity cofibers, followed by integral
singular chains.  Thus
\(\ReR\CellC_*(X)\) denotes the chain complex of the realized filtered
exact couple.  It does \emph{not} mean that a realization functor has
been applied termwise to strictly \(\Aone\)-invariant sheaves; the
statement ``the real realization of \(\HT\) is \(\Z[M]\)'' is
shorthand for the isomorphism \eqref{eq:real-HT} of realized chain
groups, not for a functor on sheaves.

Since \(U\) is split unipotent and \(T\) is split,
\(H^1(\R,U)=H^1(\R,T)=1\), and the exact sequence
\(1\to U\to B\to T\to1\) also gives \(H^1(\R,B)=1\).
Consequently the real points of the algebraic quotients used below are
the naive homogeneous spaces
\[
 (G/U)(\R)=G(\R)/U(\R),\qquad
 (G/B)(\R)=G(\R)/B(\R).
\]
Iwasawa decomposition induces a diffeomorphism
\begin{equation}\label{eq:iwasawa-X}
  (G/U)(\R)\xrightarrow{\sim}K\times A.
\end{equation}
The factor \(A\) is contractible.  The real points of a torus stratum
have \(|M|\) contractible components, so
\begin{equation}\label{eq:real-HT}
  H_0(T(\R),\Z)\cong\Z[M].
\end{equation}

The algebraic filtration is indexed by codimension, whereas the usual
Bruhat CW decomposition of \(K\) is indexed by dimension.  The reversal
\(w=w_0a\) identifies the normal cell to \(X_w\) with the opposite
Bruhat cell of \(K\) indexed by \(a\).  We use this convention below.

We first record the two pinned rank-one Iwasawa identities on which
all comparisons in this section rest; they are used in both
\Cref{lem:filtered-iwasawa} and \Cref{lem:real-orientation}.
Write \(N_u=U(\R)\) and let
\[
 r:G(\R)/N_u\longrightarrow G(\R)/AN_u\cong K
\]
be the natural quotient.

\begin{lemma}\label{lem:rank-one-iwasawa}
Let \(\alpha\) be a simple root with pinned homomorphism
\(\varphi_\alpha:SL_2(\R)\to G(\R)\).  Modulo \(AN_u\),
\begin{equation}\label{eq:rank-one-iwasawa}
 r\bigl(x_{-\alpha}(z)N_u\bigr)
 =\varphi_\alpha\!\left(
   \frac1{\sqrt{1+z^2}}
   \begin{pmatrix}1&-z\\ z&1\end{pmatrix}
  \right)AN_u,
\end{equation}
and
\begin{equation}\label{eq:positive-rank-one-iwasawa}
 r\bigl(x_{\alpha}(z)n_\alpha N_u\bigr)=
 \varphi_\alpha\!\left(
  \frac1{\sqrt{1+z^2}}
  \begin{pmatrix}z&-1\\1&z\end{pmatrix}
 \right)AN_u.
\end{equation}
\end{lemma}

\begin{proof}
For \eqref{eq:rank-one-iwasawa}, multiplying the displayed rotation by
\[
 \begin{pmatrix}
  \sqrt{1+z^2}&z/\sqrt{1+z^2}\\
  0&1/\sqrt{1+z^2}
 \end{pmatrix}\in AN_u
\]
gives \(x_{-\alpha}(z)\).  For
\eqref{eq:positive-rank-one-iwasawa}, one checks likewise that
\[
 \frac1{\sqrt{1+z^2}}
 \begin{pmatrix}z&-1\\1&z\end{pmatrix}
 \begin{pmatrix}
  \sqrt{1+z^2}&-z/\sqrt{1+z^2}\\
  0&1/\sqrt{1+z^2}
 \end{pmatrix}
 =\begin{pmatrix}z&-1\\1&0\end{pmatrix},
\]
where the middle matrix is upper triangular with positive diagonal,
hence lies in \(AN_u\), and
\(\left(\begin{smallmatrix}z&-1\\1&0\end{smallmatrix}\right)
=x_\alpha(z)n_\alpha\) by \eqref{eq:positive-rank-one}.
\end{proof}

\begin{lemma}\label{lem:filtered-iwasawa}
Let
\[
 \Phi:K\times A\xrightarrow{\ \sim\ }(G/U)(\R),
 \qquad (k,a)\longmapsto kaU,
\]
be the Iwasawa diffeomorphism, and let \(K_p^\downarrow\) be the union
of the lifted Bruhat cells with Weyl length at least \(N-p\).  For
every \(p\), \(\Phi\) restricts to a diffeomorphism of pairs
\begin{equation}\label{eq:filtered-pair-diffeomorphism}
 \bigl(K_p^\downarrow\times A,
       K_{p-1}^\downarrow\times A\bigr)
 \xrightarrow{\ \sim\ }
 \bigl(\Omega_p^X(\R),\Omega_{p-1}^X(\R)\bigr),
\end{equation}
and these diffeomorphisms commute with the inclusions in \(p\).
Projection to \(K\) is consequently a filtered strong deformation
retraction.

After the length reversal, the exact couple of
\(K_\bullet^\downarrow\) is the cellular chain complex of the lifted
opposite-Bruhat CW decomposition.  On the \(a\)-summand, the induced
comparison is
\[
 H_p\!\left(\operatorname{Re}_{\R}
   \Th(\nu_{X_{w_0a}/\Omega_p^X})\right)
 \cong H_0(T(\R),\Z)e_a
 \cong\Z[M]e_a.
\]
It is equivariant for right multiplication by \(M\).
\end{lemma}

\begin{proof}
Recall \(N_u=U(\R)\); the subscript separates this
real unipotent group from the integer \(N=\ell(w_0)\) and from the
extended Weyl group denoted by \(U\) in parts of
\cite{PatraoSandoval2026}.  The quotient
\(r:G(\R)/N_u\to G(\R)/AN_u\cong K\)
is the projection associated with \(\Phi\).  On a Bruhat stratum and
its tubular chart the algebraic coordinates are, respectively,
\[
 X_w(\R)\cong U_{w^{-1}}(\R)\times M\times A,
 \qquad
 \mathcal O_w(\R)\cong
 V_w(\R)\times U_{w^{-1}}(\R)\times M\times A.
\]
Let
\(\mathcal B_w^K=\coprod_{m\in M}\mathcal B(w,m)\) be the union of the
lifted Bruhat cells in \(K\) with Weyl index \(w\), based at the pinned
representative and carrying all \(M\)-labels.

We first prove the assertion stratum by stratum.  Order the root
coordinates according to a reduced word for \(w\), and write
\[
 \xi_{w,m}:\R^{\ell(w)}\times A\xrightarrow{\ \sim\ }(X_w(\R))_m,
 \qquad(\mathbf z,a)\longmapsto u(\mathbf z)n_wmaU,
\]
for the resulting algebraic parameterization with \(M\)-label \(m\).
Apply
\eqref{eq:rank-one-iwasawa} successively from left to right.  At one
step, the change \(z=\tan\phi\) is a diffeomorphism from the root line
to the interior of the corresponding compact characteristic
interval.  The \(AN_u\)-factor produced at that step rescales every
later root coordinate by a positive number and adds only root
coordinates of strictly greater transported height.  By
\Cref{lem:unipotent-triangular}, the resulting coordinate change is
triangular with positive diagonal.  More explicitly, in the chosen
order its \(j\)-th compact coordinate satisfies
\[
 \tan\phi_j
 =c_j(z_1,\ldots,z_{j-1})z_j+
   P_j(z_1,\ldots,z_{j-1}),\qquad c_j>0.
\]
Thus \(z_j\) is recovered uniquely and smoothly after
\(z_1,\ldots,z_{j-1}\); since
\(\tan:(-\pi/2,\pi/2)\to\R\) is onto, the recursion is global rather
than merely local.  We obtain a commutative diagram
\begin{equation}\label{eq:iwasawa-cell-diagram}
\begin{tikzcd}[column sep=large]
 \R^{\ell(w)}\times A
   \arrow[r,"\xi_{w,m}"]
   \arrow[d,"\Theta_{w,m}"']
 &(X_w(\R))_m\arrow[d,"\Phi^{-1}"]\\
 \mathring I^{\,\ell(w)}\times A
   \arrow[r,"\Psi_{w,m}\times\operatorname{id}_A"']
 &\mathcal B(w,m)\times A ,
\end{tikzcd}
\end{equation}
where \(\Psi_{w,m}\) is the Patr\~ao--Sandoval characteristic
parameterization \cite[Section~4]{PatraoSandoval2026}.  The positive
\(A\)-factor occurring in the recursive calculation is absorbed in
the second factor of the bottom row; this is a diffeomorphism because
multiplication in \(A\) is.  At the point with all root coordinates
zero, the pinned representative \(n_wm\) maps to the chosen base point
of \(\mathcal B(w,m)\).  Thus the \(M\)-label also agrees.

Taking the disjoint union of \eqref{eq:iwasawa-cell-diagram} over
\(m\in M\) proves
\begin{equation}\label{eq:filtered-iwasawa-strata}
 r\bigl(X_w(\R)\bigr)=\mathcal B_w^K,\qquad
 r^{-1}(\mathcal B_w^K)=X_w(\R)
   \cong\mathcal B_w^K\times A.
\end{equation}
Because \(\Phi\) is one global diffeomorphism, these cellwise
identities glue automatically: no independently chosen product
trivializations are being compared on overlaps.  Taking their union
over \(\ell(w)\geq N-p\) gives
\[
 \Omega_p^X(\R)=\Phi(K_p^\downarrow\times A)
\]
for every \(p\), and proves
\eqref{eq:filtered-pair-diffeomorphism} simultaneously for all \(p\).
The homotopy
\[
 H_s(k,a)=\Phi\bigl(k,\exp((1-s)\log a)\bigr),
 \qquad 0\leq s\leq1,
\]
is a strong deformation retraction through filtered maps.  Hence it
induces a morphism of the entire exact couples and commutes with every
connecting morphism, not only with the resulting homology groups.

It remains to identify the exact couple of
\(K_\bullet^\downarrow\).  In the tubular version of
\eqref{eq:iwasawa-cell-diagram}, the normal root-coordinate block is
treated by \eqref{eq:positive-rank-one-iwasawa}; the substitution
\(z=\cot(\pi t)\) identifies its one-point compactification with the
corresponding opposite characteristic interval.  Iterating in the
fixed root order identifies, for \(w=w_0a\), the normal disk to
\(\mathcal B_w^K\) with the lifted opposite Bruhat \(p\)-cell indexed
by \(a\).  Excision in a union of these tubular disks gives
\[
 H_j\!\left(
   \operatorname{hocofib}
   (K_{p-1}^\downarrow\to K_p^\downarrow);\Z
 \right)
 \cong
 \begin{cases}
   \displaystyle\bigoplus_{\ell(a)=p}\Z[M]e_a,&j=p,\\
   0,&j\ne p.
 \end{cases}
\]
Under this identification, the connecting map for the quotient triple
\[
 K_{p-2}^\downarrow\subset
 K_{p-1}^\downarrow\subset K_p^\downarrow
\]
is the cellular attaching boundary of those opposite disks: both maps
are the boundary of this same triple after excision to the disk
neighbourhoods.  This proves, rather than assumes, that the exact
couple is the lifted opposite-Bruhat cellular complex.

The right \(M\)-action preserves the filtration and permutes the
labelled disks, so the identification is right-\(\Z[M]\)-linear.
The endpoint and determinant orientations are computed in
\Cref{lem:real-orientation}.  Finally \(T(\R)=M\times A\); contraction
of \(A\) identifies the realized torus summand with \(\Z[M]\), giving
the displayed comparison on each \(a\)-summand.
\end{proof}

We recall the group-ring form of the Patr\~ao--Sandoval boundary
\cite[Theorems~1.2 and~6.4]{PatraoSandoval2026}.\footnote{In the first
display for \(\Psi_u(I_i^{d-1})\) in
\cite[Lemma~4.6]{PatraoSandoval2026}, we interpret the symbol between
\(\mathcal B(u_i^0)\) and \(\mathcal B(u_i^1)\) as union, consistently
with the two cells being disjoint and with the later skeleton formula.
The boundary formula used here is unaffected.}  Let
\[
  a=s_{i_1}\cdots s_{i_p}
\]
and delete position \(q\).  Move \(c_{i_q}\) through the tail to the
right.  The resulting element is
\[
  c_{a,q}=\lambda_{a,q}(-1)\in M.
\]
If \(\epsPS(a,q)\in\{\pm1\}\) is the ordinary degree comparing the
deleted-word target orientation with the chosen target orientation,
the two endpoint faces have the individual incidence numbers
\begin{equation}\label{eq:PS-two-faces}
 \rho(a,a_q^0)=\epsPS(a,q)(-1)^q,
 \qquad
 \rho(a,a_q^1)=\epsPS(a,q)(-1)^{q+1+\sigma_{a,q}}.
\end{equation}
Here the second target lift is the right translate of the first,
\(e_{a_q^1}=e_{a_q^0}c_{a,q}\).  Consequently the two faces combine,
relative to the first target lift, to
\begin{equation}\label{eq:PS-group-ring}
  \epsPS(a,q)(-1)^q
  \left(1-(-1)^{\sigma_{a,q}}c_{a,q}\right)\in\Z[M].
\end{equation}

\begin{lemma}\label{lem:real-orientation}
Under the filtered Iwasawa comparison, the normal-word factor and the
ordinary lifted-cell orientation factor satisfy
\[
 \operatorname{sign}\bigl(\vartheta_{a,q}\bigr)=\epsPS(a,q).
\]
Here \(\operatorname{sign}:\GW(\R)^\times\to\{\pm1\}\) is the
signature on one-dimensional forms.
\end{lemma}

\begin{proof}
Put \(d=p-1\).  For a reduced word
\(\mathbf c=(j_1,\ldots,j_d)\), left translation by
\(n_{w_0}^{-1}\) carries the actual opposition chart to
\begin{equation}\label{eq:positive-real-chart}
 P_{\mathbf c}(\mathbf z)=
 x_{\alpha_{j_1}}(z_1)n_{j_1}\cdots
 x_{\alpha_{j_d}}(z_d)n_{j_d}
\end{equation}
by \eqref{eq:opposition-bott-samelson}.  This left translation commutes
with every right \(M\)-translation, so it preserves the labels of the
two lifted faces.

Write \(\psi_i(t)=\exp(tF_i)\) and
\[
 \Psi_{\mathbf c}(t_1,\ldots,t_d)
   =\psi_{j_1}(t_1)\cdots\psi_{j_d}(t_d)AN_u
   \quad(0\leq t_\nu\leq1)
\]
for the Patr\~ao--Sandoval characteristic cube attached to
\(\mathbf c\); below we use its interior near the common midpoint.  We
first record the orientation of the Iwasawa
coordinate change.  In rank one, by
\eqref{eq:positive-rank-one-iwasawa} of \Cref{lem:rank-one-iwasawa},
with the preceding choice \(s_i^{\mathrm{PS}}=n_i\), the right side is the
Patr\~ao--Sandoval path \(\psi_i(t)\) with
\(z=\cot(\pi t)\), up to a positive rescaling coming from their
normalization of \(E_{\alpha_i}\).  At the midpoint
\(t=1/2\), therefore, \(dt/dz<0\).

Iterating \eqref{eq:positive-rank-one-iwasawa} along a reduced word
defines a local coordinate change
\[
 J_{\mathbf c}:\mathbf z\longmapsto\mathbf t,
 \qquad
 r(P_{\mathbf c}(\mathbf z)N_u)
   =\Psi_{\mathbf c}(J_{\mathbf c}(\mathbf z)).
\]
This is not an independently chosen comparison of parameterizations:
it is the restriction of the global filtered diffeomorphism
\(\Phi^{-1}\) in \eqref{eq:filtered-pair-diffeomorphism} to the
normal disk at the indicated stratum.  Its derivative is therefore
the normal-bundle map inducing the Thom-space comparison in
\Cref{lem:filtered-iwasawa}.
The sign can be read directly in the tangent quotient.  Put
\[
 \beta_\nu=s_{j_1}\cdots s_{j_{\nu-1}}(\alpha_{j_\nu}).
\]
At the common midpoint \(n_cAN_u\), the derivative
of the \(\nu\)-th algebraic coordinate is represented by
\[
 \operatorname{Ad}(n_{<\nu})E_{\alpha_{j_\nu}},
\]
whereas the corresponding compact cube derivative is represented by
\[
 \pi\operatorname{Ad}(n_{<\nu})
 (E_{-\alpha_{j_\nu}}-E_{\alpha_{j_\nu}}).
\]
Because \(\beta_\nu\) is an inversion root of the reduced word
\(\mathbf c\), one has \(c^{-1}\beta_\nu\in\Phi^-\), and hence
\(c^{-1}(-\beta_\nu)\in\Phi^+\).  The transported negative-root term
therefore lies in the isotropy algebra
\(\operatorname{Ad}(n_c)(\mathfrak a+\mathfrak n_u)\) and vanishes in
\(T_{n_cAN_u}(G/AN_u)\).  Thus the compact derivative is \(-\pi\) times
the algebraic derivative on every ordered coordinate.  Equivalently,
\(DJ_{\mathbf c}\) has negative diagonal and positive determinant
factor apart from its signs.  Hence
\begin{equation}\label{eq:J-word-sign}
 \sgn\det(DJ_{\mathbf c})=(-1)^d
\end{equation}
for every reduced word of length \(d\).  The sign in
\eqref{eq:J-word-sign} results from retaining the actual root vector
\(\operatorname{Ad}(n_{w_0})E_\alpha\) rather than replacing it with a
fixed negative-root vector.

There is a second, absolute sign, coming from the Thom generator.  We
use the realization convention in which the one-coordinate motivic
Thom class \(\tau_z\) has connecting class
\begin{equation}\label{eq:real-thom-convention}
 \partial\tau_z=[z<0]-[z>0]
 \quad\text{in}\quad
 \widetilde H_0(\R^\times,\Z).
\end{equation}
This is the convention of \Cref{prop:D-real}: the negative component is
the symbol \((-1)\), and the positive component is the unit
basepoint.  Relative to the usual \(dz\)-oriented topological Thom
class, \(\tau_z\) has sign \(-1\).  The \(p\)-fold motivic Thom
generator therefore contributes \((-1)^p\), while the Iwasawa
coordinate change for a word of length \(p\) contributes the same
\((-1)^p\) by \eqref{eq:J-word-sign}.  These signs cancel.  Thus the
realized motivic basis maps to the standard oriented
Patr\~ao--Sandoval cube basis without a residual degree-dependent sign.
In rank one this normalization says
\[
 \partial\tau_z=[\lambda(-1)]-[1]=c_i-1,
\]
which is also the boundary of their arc oriented from \(t=0\) to
\(t=1\).  This \(SL_2\) calculation fixes the absolute sign of the
descending-open exact couple in every degree: the ordered \(p\)-fold
Thom and cube models are tensor products of this one-line model, and
their total-complex signs are the same signs computed in
\Cref{lem:face-sign}.  Thus no degree-\(p\) sign is inferred merely
from the rank-one numerical boundary.

Now let
\(g=\Phi_{\mathbf b}^{-1}
       \Phi_{\widehat{\mathbf a}_q}\).
Its Jacobian determinant is the constant whose quadratic class is
\(\vartheta_{a,q}\).  Choose a common regular point of the two target
cell charts.  In local cube coordinates the transition used in the
definition of \(\epsPS(a,q)\) is, by the chain rule,
\[
 J_{\mathbf b}\circ g\circ
 J_{\widehat{\mathbf a}_q}^{-1}.
\]
Both words have length \(d\), so the two signs
\((-1)^d\) in \eqref{eq:J-word-sign} cancel.  Its local degree is
therefore \(\sgn(\det Dg)\), independently of the chosen regular
point because \(\det Dg\) is constant.  Right translation by the
second face label has the same coordinate comparison.  We unwind the
cited definition as follows.  In
\cite[Section~5]{PatraoSandoval2026}, the sign \(\epsPS(a,q)\) is the
degree of the attaching comparison between the face of the
characteristic cube of \(\mathbf a\) obtained by setting
\(t_q\in\{0,1\}\) and the characteristic cube of the chosen word for
\(b\): both parametrize the same lifted cell, and the degree is the
sign of the Jacobian of the transition between the two cube
parametrizations at any interior point.  The face restriction of
\(\Psi_{\mathbf a}\) is exactly \(\Psi_{\widehat{\mathbf a}_q}\)
translated by the face label, so this transition is the displayed
composite \(J_{\mathbf b}\circ g\circ
J_{\widehat{\mathbf a}_q}^{-1}\) up to that translation, which
preserves orientation.  Its degree
is therefore \(\sgn(\det Dg)\), and by
\cite[Theorem~6.4]{PatraoSandoval2026} this degree is the
coefficient sign \(\epsPS(a,q)\).  Finally
\(\operatorname{sign}\ang{\det Dg}=\sgn(\det Dg)\), which proves the
claim.
\end{proof}

\begin{remark}\label{rem:sign-dictionary}
For reference, the complete dictionary between the motivic and
Patr\~ao--Sandoval sign data is:
\[
\begin{array}{p{0.42\textwidth}p{0.5\textwidth}}
\toprule
\text{motivic datum} & \text{Patr\~ao--Sandoval datum}\\
\midrule
frame degree \(\vartheta_{a,q}\in\GW(\R)^\times\)
 & orientation sign \(\epsPS(a,q)=\operatorname{sign}\vartheta_{a,q}\)\\
ordinary face sign \((-1)^{q-1}\)
 & face sign \((-1)^{q}\) (opposite convention, absorbed by the
   endpoint order below)\\
determinant weight \((-1)^{\sigma_{a,q}}\)
 & relative sign of the two faces in \eqref{eq:PS-two-faces}\\
torus translation \([\lambda_{a,q}(u)t]\)
 & second face label \(e_{a_q^1}=e_{a_q^0}c_{a,q}\)\\
Thom convention \eqref{eq:real-thom-convention}
 & cube orientation from \(t=0\) to \(t=1\)\\
\bottomrule
\end{array}
\]
All five rows are established inside the proofs of
\Cref{lem:filtered-iwasawa,lem:real-orientation}; no sign is imported
from \cite{PatraoSandoval2026} without an explicit coordinate
identification.
\end{remark}

\begin{theorem}\label{thm:real-comparison}
There is an isomorphism of right-\(\Z[M]\) chain complexes
\begin{equation}\label{eq:real-chain-iso}
  \ReR\CellC_*(G/U)\xrightarrow{\sim}
  C_*^{\mathrm{CW}}(K;\Z),
\end{equation}
where the right side is the extended-Weyl Bruhat complex of
\cite{PatraoSandoval2026}, written with opposite cells and the
\(w_0\)-reindexing.  Under this isomorphism, every motivic cover
component in \eqref{eq:main-boundary} becomes
\eqref{eq:PS-group-ring}.
\end{theorem}

\begin{proof}
\Cref{lem:filtered-iwasawa} gives an isomorphism of the filtered exact
couples, so it already commutes with the connecting morphisms.
\Cref{lem:real-orientation} identifies the Thom bases under this
isomorphism with the oriented lifted-cell bases.  It remains only to
express the common connecting morphism in those bases.  By
\Cref{prop:D-real}, the universal operator realizes to
\[
  (-1)^{\sigma_{a,q}}c_{a,q}-1.
\]
Including the face sign gives
\[
  \epsPS(a,q)(-1)^{q-1}
  \left((-1)^{\sigma_{a,q}}c_{a,q}-1\right)
  =
  \epsPS(a,q)(-1)^q
  \left(1-(-1)^{\sigma_{a,q}}c_{a,q}\right),
\]
which is \eqref{eq:PS-group-ring}.  Hence the comparison commutes with
every differential.
\end{proof}

\begin{corollary}
\label{cor:compact-fundamental-cycle}
Let \(N=\ell(w_0)\); under the standing split-real hypothesis of this
section one has \(N=|\Phi^+|=\dim K\).  Put
\[
 \mathcal N_M=\sum_{m\in M}m\in\Z[M].
\]
In the extended-Weyl complex, the chain
\(\mathcal N_M e_{w_0}\in C_N^{\mathrm{CW}}(K;\Z)\) is a cycle and,
up to the global orientation sign, represents the fundamental class
\([K]\).  Under the quotient \(K\to K/M\), it maps to
\[
 |M|e_{w_0},
\]
so the comparison square records the degree \(|M|\) of the compact
covering at chain level.
\end{corollary}

\begin{proof}
For a top cover, the corresponding ordinary Bruhat root is simple.
Hence \Cref{lem:sigma-height} gives \(\sigma=0\), and every realized
cover entry is, up to its orientation sign, right multiplication by
\(c-1\) for some \(c\in M\).  Since
\(\mathcal N_Mc=\mathcal N_M\), every component of the boundary of
\(\mathcal N_Me_{w_0}\) vanishes.  The right translations permuting the
top cells preserve their compatible orientations: \(K\) is connected
under the standing hypothesis, and every right translation is
therefore isotopic to the identity.  Thus this is the
cellular fundamental cycle of the connected compact manifold \(K\).
Group-ring augmentation sends \(\mathcal N_M\) to \(|M|\), proving the
last statement.
\end{proof}

\section{Torus augmentation and the flag complex}
\label{sec:augmentation}

The right quotient of \(X=G/U\) by \(T\) is \(G/B\).  Morel--Sawant
proved that the cellular quotient is computed by torus augmentation
\cite[Section~5]{MorelSawant2023}; we use the orientations pulled back
from the same fixed pinning and normal words.

\begin{theorem}\label{thm:flag-augmentation}
There is a natural chain isomorphism
\begin{equation}\label{eq:flag-base-change}
  \CellC_*(G/U)\otimes_{\HT}\Z
  \xrightarrow{\sim}\CellC_*(G/B).
\end{equation}
Under this isomorphism, the component of a cover \(b\lessdot a\) is
\begin{equation}\label{eq:flag-cover}
  (-1)^{q-1}\vartheta_{a,q}
  \big((u)\otimes x\longmapsto
       (\ang{u^{\sigma_{a,q}}}-1)x\big).
\end{equation}
Equivalently, it is the Milnor--Witt degree of the
\(\sigma_{a,q}\)-power map followed by \(\eta\).
\end{theorem}

\begin{proof}
We argue as in \Cref{thm:G-X}, in the opposite direction.  The
projection \(\pi:G/U\to G/B\) is a Zariski locally trivial
\(T\)-torsor carrying the filtration \(\Omega_\bullet^X\) to the
Bruhat filtration of \(G/B\), stratum by stratum:
\(\pi\) restricts on \(X_w\cong\mathbb A^{\ell(w)}\times T\) to the
projection onto the affine Bruhat cell.  On tubular charts, \(\pi\) is
the map \((v,u,t)\mapsto(v,u)\), which identifies the source normal
frame with the pullback of the target normal frame; the pinned
orientations correspond.  Naturality of purity for these transverse
smooth squares, in the precise form of
\Cref{lem:purity-naturality}, gives a morphism of exact couples, hence a chain map
which on the \(w\)-summand is
\[
 \KMW p\otimes\HT
 =\KMW p\otimes\ZA[T]
 \longrightarrow\KMW p\otimes\Z=\KMW p,
\]
induced by \(T\to\operatorname{pt}\), i.e.\ the torus augmentation
\(\epsilon_T\) tensored with the identity.  Since \(\epsilon_T\) is
the cokernel of the inclusion of the augmentation ideal, and the
differentials are right \(\HT\)-linear (\Cref{thm:all-degree}), the
chain map factors through an isomorphism from
\(\CellC_*(G/U)\otimes_{\HT}\Z\), on every chain group and every
connecting morphism.  This proves \eqref{eq:flag-base-change}, and
\Cref{prop:D-augmentation} gives \eqref{eq:flag-cover}.
\end{proof}

\begin{remark}
\label{rem:flag-unconditional}
Formula \eqref{eq:flag-cover} computes the flag differential in
\emph{every} cellular degree and every root type, with no
boundary-data hypothesis: the inputs \(\sigma_{a,q}\) and
\(\vartheta_{a,q}\) are computed by
\Cref{lem:sigma-height,prop:frame-algorithm,prop:rank-two-degrees}.
The hypothesis of \cite[Definition~2.1]{LiuLiu2026} enters only when
one wishes to \emph{rewrite} \eqref{eq:flag-cover} in terms of the
half-boundary matrices \(E_i\) constructed there, as in
\Cref{cor:eta}: the matrices \(E_i\) are defined from signed
boundary data available in that range.  In this precise sense the
present paper removes the range restriction from the flag
differential itself, while statements phrased in terms of \(E_i\)
retain it.
\end{remark}

\begin{proposition}
\label{prop:flag-orientation-dictionary}
Assume the boundary-data hypothesis of \cite[Definition~2.1]{LiuLiu2026}
and choose its normalized Chevalley coordinates and fixed normal-form
words.  For every reduced deletion, the frame factor in this paper is
the local coordinate-degree factor used there:
\[
 \vartheta_{a,q}
 =\deg^{\Aone}\!\left(
    \Phi_{\mathbf b}^{-1}\Phi_{\widehat{\mathbf a}_q}
   \right).
\]
In particular, this equality identifies the full
Grothendieck--Witt unit, not only its real signature or parity.
\end{proposition}

\begin{proof}
Both sides are the Milnor--Witt degree of the same transition from the
deleted-word root-product chart to the fixed target normal-form chart.
The determinant-line convention is the same complementary orientation
\[
 \det T(C_w)\otimes\det\nu_{C_w}
 \cong\det T(G/B)|_{C_w};
\]
compare \cite[Lemma~2.4]{LiuLiu2026}.  In the normalized rank-two
coordinates, commutations, short \(A_2\) braids, and long \(B_2\)
braids have degrees \(\ang{-1}\), \(\ang1\), and \(\ang{-1}\),
respectively, and their products give the intrinsic coordinate degree
\cite[Lemmas~4.2 and~4.5]{LiuLiu2026}.  These are precisely the
elementary transitions whose composite defines
\(\vartheta_{a,q}\) in \eqref{eq:theta-word}.  Hence the two
Grothendieck--Witt classes agree.
\end{proof}

The root-theoretic parity identity used in real flag boundaries says
\begin{equation}\label{eq:parity}
  \sigma_{a,q}\equiv
  \operatorname{ht}(\gamma_{a,q}^{\vee})+1\pmod2,
\end{equation}
where \(\gamma_{a,q}\) is the positive root of the associated ordinary
Bruhat cover.  It follows either by comparing the tail inversion set
with the coroot expansion, or from the usual root-height boundary
formula.

\begin{corollary}\label{cor:eta}
In the boundary-data range of \cite{LiuLiu2026}, and after its stated
normal-word reindexing and diagonal orientation choices,
\eqref{eq:flag-cover} is
\[
  \partial_i^{\Aone}=\eta E_i\qquad(i\geq2),
  \qquad \partial_1^{\Aone}=0.
\]
Thus the half-boundary matrix of the flag variety is obtained by
augmenting the torus-enriched group complex.
\end{corollary}

\begin{proof}
The conclusion is motivic, not an inference from real degree.
\Cref{prop:flag-orientation-dictionary} identifies the complete
normal-form unit with that of \cite[Theorem~4.11]{LiuLiu2026}.  It
therefore remains only to identify the universal power factor.  In
Milnor--Witt \(K\)-theory,
\[
 \ang{u^m}-1=\eta(u^m)=
 \begin{cases}
   0,&m\ \text{even},\\
   \eta(u),&m\ \text{odd},
 \end{cases}
\]
also for negative \(m\), because \(\ang{u^{-1}}=\ang u\).
Equivalently this is the identity
\(m_\eps\eta=0\) for even \(m\) and
\(m_\eps\eta=\eta\) for odd \(m\), as recorded in \Cref{sec:prelim}.
Thus an even \(\sigma_{a,q}\) gives zero, while an odd one gives a unit
multiple of \(\eta\).  The unit is exactly the
normal-word degree \(\vartheta_{a,q}\), and
\eqref{eq:parity} identifies the nonzero cases and signs with the
reindexed half-boundary entries.
\end{proof}

\begin{theorem}\label{thm:comparison-square}
For a Bruhat-filtered real variety \(Y\), write
\(C_*^{\mathrm{fil}}(\operatorname{Re}_{\R}Y;\Z)\) for the integral
chain complex of the exact couple of its realized filtration.  If
\(k=\R\), there is a commutative square in \(K(\mathbf{Ab})\):
\[
\begin{tikzcd}[column sep=large,row sep=large]
 C_*^{\mathrm{fil}}(\operatorname{Re}_{\R}(G/U);\Z)
   \arrow[r,"\pi_*"]
   \arrow[d,"\simeq"']
 &
 C_*^{\mathrm{fil}}(\operatorname{Re}_{\R}(G/B);\Z)
   \arrow[d,"\simeq"]
 \\
 C_*^{\mathrm{CW}}(K;\Z)
   \arrow[r,"-\otimes_{\Z[M]}\Z"']
 &
 C_*^{\mathrm{CW}}(K/M;\Z).
\end{tikzcd}
\]
Here the top arrow is induced by the filtered quotient
\(\pi:G/U\to G/B\) before taking singular chains.  Under the stratum
identifications it is extension of scalars
\(\Z[M]\to\Z\); the top row is not a termwise realization of a complex
of strictly \(\Aone\)-invariant sheaves.
The left vertical map retains the two \(M\)-labelled faces; either
horizontal map identifies them.
\end{theorem}

\begin{proof}
Iwasawa decomposition gives simultaneous diffeomorphisms
\[
 (G/U)(\R)\cong K\times A,\qquad
 (G/B)(\R)\cong K/M.
\]
Because \(T(\R)=M\times A\) and \(M\) centralizes \(A\), the real
quotient map is, in these coordinates,
\[
 \pi(\R):(k,a)\longmapsto kM.
\]
The stratum diagram \eqref{eq:iwasawa-cell-diagram} and its quotient
by \(M\) identify the filtrations, for every \(p\), as
\[
 \Omega_p^{G/U}(\R)=K_p^\downarrow\times A,
 \qquad
 \Omega_p^{G/B}(\R)=K_p^\downarrow/M.
\]
Consequently there is a strictly commutative square of filtered pairs
\begin{equation}\label{eq:filtered-comparison-square}
\begin{tikzcd}[column sep=large,row sep=large]
 (K_p^\downarrow\times A,K_{p-1}^\downarrow\times A)
   \arrow[r,"q_p"]
   \arrow[d,"r_p"']
 &
 (K_p^\downarrow/M,K_{p-1}^\downarrow/M)
   \arrow[d,"\mathrm{id}"']
 \\
 (K_p^\downarrow,K_{p-1}^\downarrow)
   \arrow[r]
 &
 (K_p^\downarrow/M,K_{p-1}^\downarrow/M).
\end{tikzcd}
\end{equation}
These squares commute with the inclusions in \(p\).

Take homotopy cofibers of the inclusions of pairs in
\eqref{eq:filtered-comparison-square}, then integral singular chains
and the associated exact couples.  Functoriality of each of these
operations gives a commutative square of chain complexes.  On the
left, the vertical map is the filtered equivalence of
\Cref{lem:filtered-iwasawa}.  On the right, the opposite normal disk
indexed by \((a,m)\) maps homeomorphically to the flag disk indexed by
\(a\); we orient the latter by this quotient map.  Hence the right
vertical identification is compatible with the left one and with all
connecting morphisms.

On relative chain groups, the top horizontal map sends every
\(M\)-labelled generator \(me_a\) to the single generator \(e_a\).
It is therefore the augmentation
\(\Z[M]\to\Z\).  The bottom horizontal map is the cellular quotient by
the same free right \(M\)-action, hence is
\(-\otimes_{\Z[M]}\Z\).  This proves the displayed square, in fact
strictly after the stated basis choices and therefore also in
\(K(\mathbf{Ab})\).  For a cover, the equality between
\eqref{eq:D-real} after augmentation and the realization of
\eqref{eq:D-aug} is a componentwise consequence of this filtered
square, rather than a substitute for its geometric construction.
\end{proof}

\section{Examples}
\label{sec:examples}

\subsection{Rank one}

Let \(G=SL_2\), \(T\cong\Gm\), and \(X=G/U\cong\mathbb A^2-\{0\}\).
There are two Weyl elements and
\[
  \CellC_1(X)=\KMW1\otimes\HT,\qquad
  \CellC_0(X)=\HT.
\]
The unique differential is
\[
  \Dop_{\alpha^\vee,0}^{(1)}:
  (u)\otimes[t]\longmapsto
  [\alpha^\vee(u)t]-[t].
\]
Using \(\HT=\KMW1\oplus\Z\), this is
\begin{equation}\label{eq:SL2-matrix}
  \KMW2\oplus\KMW1
  \xrightarrow{
  \left(\begin{smallmatrix}\eta&\operatorname{id}\\
                            0&0
         \end{smallmatrix}\right)}
  \KMW1\oplus\Z.
\end{equation}
Its kernel is the image of
\[
  \KMW2\longrightarrow\KMW2\oplus\KMW1,\qquad
  x\longmapsto(x,-\eta x).
\]
Thus \(\CellH_1(SL_2)\cong\KMW2\), in agreement with
\cite{MorelSawant2023}.  Real realization is the
\(\Z[\Z/2]\)-map \(c-1\), the boundary of the two-vertex, two-edge
circle.  Augmentation is zero, the cellular differential of
\(\mathbb P^1\).

\subsection{\texorpdfstring{\(SL_3\)}{SL3}}

Let the simple roots be \(\alpha_1,\alpha_2\), write
\(\alpha_{12}^\vee=\alpha_1^\vee+\alpha_2^\vee\), and choose
\[
 e,\quad s_1,\quad s_2,\quad s_1s_2,\quad s_2s_1,\quad
 w_0=s_1s_2s_1
\]
as normal words.  Order the summands by
\[
 C_1:(s_1,s_2),\qquad
 C_2:(s_1s_2,s_2s_1).
\]
For compact notation, \(\Dop_{\lambda,m}^{[p]}\) denotes the operator in
cellular degree \(p\), including the evident cut of the chosen symbol
factor.  The full complex is
\begin{equation}\label{eq:SL3-complex}
\begin{split}
 \KMW3\otimes\HT
 &\xrightarrow{\partial_3}
 (\KMW2\otimes\HT)^{\oplus2}
 \xrightarrow{\partial_2}
 (\KMW1\otimes\HT)^{\oplus2}
 \xrightarrow{\partial_1}\HT,\\[2mm]
 \partial_1&=
 \begin{pmatrix}
   \Dop_{\alpha_1^\vee,0}^{[1]}&
   \Dop_{\alpha_2^\vee,0}^{[1]}
 \end{pmatrix},\\[2mm]
 \partial_2&=
 \begin{pmatrix}
  -\Dop_{\alpha_2^\vee,0}^{[2]}&
   \Dop_{\alpha_{12}^\vee,-1}^{[2]}\\
   \Dop_{\alpha_{12}^\vee,-1}^{[2]}&
  -\Dop_{\alpha_1^\vee,0}^{[2]}
 \end{pmatrix},\\[2mm]
 \partial_3&=
 \begin{pmatrix}
   \Dop_{\alpha_1^\vee,0}^{[3]}\\
   \Dop_{\alpha_2^\vee,0}^{[3]}
 \end{pmatrix}.
\end{split}
\end{equation}
For these normal words, every reduced deletion yields the chosen normal
word for its target.  Hence every frame factor \(\vartheta_{a,q}\) occurring
in this complex is \(\ang1\).
The middle off-diagonal operator has weight \(-1\) because moving the
deleted simple coroot through the other simple root has determinant
exponent \(-1\).
To display all cuts without ambiguity, put
\[
 x_i(r)=[\alpha_i^\vee(r)]-[1],\qquad
 g_i(r)=1+x_i(r)=[\alpha_i^\vee(r)]
\]
and extend \(x_i:\KMW1\to\HT\) linearly.  Set
\[
 R(r,s)=\ang r(s)\otimes g_1(r)g_2(r)-(s)\otimes1.
\]
Thus \(R=\Dop_{\alpha_{12}^\vee,-1}^{[2]}\).  With a final right
factor \([t]\) understood, the four degree-two cuts are
\begin{equation}\label{eq:SL3-cut-table}
\begin{array}{ccc}
\toprule
 &\text{target }s_1&\text{target }s_2\\
\midrule
\text{source }s_1s_2&-(u)\otimes x_2(v)&R(u,v)\\
\text{source }s_2s_1&R(u,v)&-(u)\otimes x_1(v)\\
\bottomrule
\end{array}
\end{equation}
The two top cuts are
\begin{align}
 \bigl(\partial_3\bigr)_{s_1s_2}
   ((u)(v)(z)\otimes[t])
   &=(u)(v)\otimes x_1(z)[t],\label{eq:SL3-top-cut-1}\\
 \bigl(\partial_3\bigr)_{s_2s_1}
   ((u)(v)(z)\otimes[t])
   &=(v)(z)\otimes x_2(u)[t].\label{eq:SL3-top-cut-2}
\end{align}

\begin{proposition}
\label{prop:SL3-square-zero}
The maps in \eqref{eq:SL3-complex} satisfy
\(\partial_1\partial_2=0\) and \(\partial_2\partial_3=0\) directly in
the category of strictly \(\Aone\)-invariant sheaves.
\end{proposition}

\begin{proof}
The identities
\begin{equation}\label{eq:SL3-elementary-identities}
 g_i(r)x_i(s)=x_i(\ang r(s)),\qquad
 g_i(r)^2x_i(s)=x_i(s)
\end{equation}
follow from the Hopf multiplication on \(\ZA[\Gm]\) and
\(\ang{r^2}=1\).  On the first column of \(\partial_2\), the diagonal
contribution to \(\partial_1\partial_2\) is
\(-x_1(u)x_2(v)\), while the other contribution is
\begin{align*}
 x_2(\ang u(v))g_1(u)g_2(u)-x_2(v)
 &=g_2(u)x_2(v)g_1(u)g_2(u)-x_2(v)\\
 &=x_1(u)x_2(v).
\end{align*}
They cancel.  Interchanging \(1\) and \(2\) proves the assertion for
the second column.

Right \(\HT\)-linearity reduces the two rows of
\(\partial_2\partial_3\), on \((u)(v)(z)\), to
\begin{align*}
 &-(u)\otimes x_2(v)x_1(z)+R(v,z)x_2(u),\\
 &R(u,v)x_1(z)-(v)\otimes x_1(z)x_2(u).
\end{align*}
Using \eqref{eq:SL3-elementary-identities} twice gives
\[
 R(v,z)x_2(u)=(z)\otimes x_1(v)x_2(u),\qquad
 R(u,v)x_1(z)=(v)\otimes x_1(z)x_2(u).
\]
The second row therefore cancels literally.  In the first row, the two
remaining maps send \((u)(v)(z)\) to
\[
 (u)\otimes x_1(z)x_2(v),\qquad
 (z)\otimes x_1(v)x_2(u).
\]
We track how these arise from the two orders of deletion, making the
cut convention of \Cref{constr:cut} explicit.  The top word is
\(\mathbf w_0=(1,2,1)\) with ordered symbol factors
\((u)(v)(z)\), and the two paths through the diamond
\(121\to\{12,21\}\to\{1\}\) consume the three factors in different
orders: on the first path the surviving \(\KMW1\)-factor is \((u)\)
and the consumed factors produce, through
\Cref{constr:cut}, the ordered torus labels \(x_1(z)x_2(v)\); on the
second path the survivor is \((z)\) with labels \(x_1(v)x_2(u)\).
Relative to the identity order, the two composites realize the
factor permutations \((23)\) and \((13)\) of
\((u),(v),(z)\), respectively.  Both are odd, so by the
\(\eps\)-graded commutativity of \Cref{lem:MW-kunneth} both displayed
maps equal one and the same morphism, namely the identity-order
map composed with a single factor
\(\eps=-\ang{-1}\); in particular they are \emph{equal} as morphisms
\(\KMW3\to\KMW1\otimes\HT\), and the torus labels agree because the
two paths transport the same two coroots.  Meanwhile the composite
ordinary face signs of the two paths are opposite, so the two equal
maps enter
\(\partial_2\partial_3\) with opposite matrix signs and cancel
exactly.  This proves both composites
are zero as sheaf
morphisms, before augmentation or realization.
\end{proof}

\begin{proposition}
\label{prop:SL3-top-check}
The two components of \(\partial_3\) in
\eqref{eq:SL3-complex} also follow directly from purity for
\[
  B\subset P_i=B\sqcup Bs_iB\subset SL_3.
\]
\end{proposition}

\begin{proof}
The normal bundle of \(P_i\) in \(SL_3\) is trivialized by left
translation.  For the flag of regular immersions above, transitivity of
purity splits the Thom space into the rank-two normal factor of
\(P_i\subset SL_3\) and the rank-one normal factor of
\(B\subset P_i\).  The connecting morphism is the identity on the first
factor and the rank-one coroot operator on the second.  This gives the
two top cuts for the words \(121\) and \(212\) in
\eqref{eq:SL3-top-cut-1} and \eqref{eq:SL3-top-cut-2}.

It remains to compare the determinant orientations.  For the standard
matrix pinning, the normal root frames associated with
\(s_1s_2s_1\) and \(s_2s_1s_2\), respectively, are
\[
  (E_{21},-E_{31},E_{32})
  \quad\text{and}\quad
  (E_{32},E_{31},E_{21}).
\]
Both have determinant \(-1\) relative to
\((E_{21},E_{31},E_{32})\).  Thus the \(A_2\) braid introduces no
additional Milnor--Witt degree, and the direct purity calculation agrees
with \eqref{eq:SL3-complex}.
\end{proof}

Applying \(\HT\to\Z\) kills every weight-zero entry.  In the
normal-word bases of \eqref{eq:SL3-complex}, the two surviving
off-diagonal entries are \(+\eta\).  Negating both degree-two
generators gives the standard flag generators used in
\cite{LiuLiu2026}, in which the complex becomes
\[
  \partial_2^{\Aone}=
  \begin{pmatrix}0&-\eta\\-\eta&0\end{pmatrix},
  \qquad \partial_1^{\Aone}=\partial_3^{\Aone}=0,
\]
which is the \(SL_3/B\) complex in \cite{LiuLiu2026}.

For real realization, put
\(c_i=\alpha_i^\vee(-1)\) and
\(c_{12}=c_1c_2\).  The operators in
\eqref{eq:SL3-complex} give
\begin{align*}
 \delta_1(s_1)&=c_1-1,&
 \delta_1(s_2)&=c_2-1,\\
 \delta_2(s_1s_2)&=s_1(1-c_2)-s_2(1+c_1c_2),\\
 \delta_2(s_2s_1)&=s_2(1-c_1)-s_1(1+c_1c_2),\\
 \delta_3(s_1s_2s_1)&=
 s_1s_2(c_1-1)+s_2s_1(c_2-1).
\end{align*}
These are exactly the Patr\~ao--Sandoval matrices for
\(K=SO(3)\).  Their homology is
\[
  H_0(K;\Z)=\Z,\qquad H_1(K;\Z)=\Z/2,\qquad
  H_2(K;\Z)=0,\qquad H_3(K;\Z)=\Z.
\]

\subsection{\texorpdfstring{\(SL_4\)}{SL4}}
\label{subsec:SL4}

\begin{proposition}
\label{prop:SL4-degree-three}
Write \(1,2,3\) for the simple reflections and
\(\alpha_I^\vee=\sum_{i\in I}\alpha_i^\vee\).  Choose the
lexicographically first reduced word for each Weyl element.  In degrees
three and two these words, in the order used below, are
\[
 (121,123,132,213,232,321),
 \qquad (12,13,21,23,32).
\]
All reduced deletions, and hence all nonzero components of
\(\partial_3\) for \(SL_4/U\), are listed in the following table.
An entry in the row \((a,q,b,\lambda,\sigma,\vartheta)\) denotes
\[
 (-1)^{q-1}\vartheta\,
 \Dop_{\lambda,\sigma}^{(3)}\operatorname{cut}_{a,q}:
 \KMW3\otimes\HT\longrightarrow\KMW2\otimes\HT.
\]
\[
\begin{array}{c c c c r c}
\toprule
a&q&b&\lambda&\sigma&\vartheta\\
\midrule
121&1&21&\alpha_2^\vee&0&\ang1\\
121&3&12&\alpha_1^\vee&0&\ang1\\
123&1&23&\alpha_{123}^\vee&-2&\ang1\\
123&2&13&\alpha_{23}^\vee&-1&\ang1\\
123&3&12&\alpha_3^\vee&0&\ang1\\
132&1&32&\alpha_{12}^\vee&-1&\ang1\\
132&2&12&\alpha_{23}^\vee&-1&\ang1\\
132&3&13&\alpha_2^\vee&0&\ang1\\
213&1&13&\alpha_{123}^\vee&-2&\ang1\\
213&2&23&\alpha_1^\vee&0&\ang1\\
213&3&21&\alpha_3^\vee&0&\ang1\\
232&1&32&\alpha_3^\vee&0&\ang1\\
232&3&23&\alpha_2^\vee&0&\ang1\\
321&1&21&\alpha_{123}^\vee&-2&\ang1\\
321&2&13&\alpha_{12}^\vee&-1&\ang{-1}\\
321&3&32&\alpha_1^\vee&0&\ang1\\
\bottomrule
\end{array}
\]
After torus augmentation, the complete degree-three flag differential is
\begin{equation}\label{eq:SL4-flag-d3}
 \partial_3^{\Aone}(SL_4/B)
 =\eta
 \begin{pmatrix}
  0&0&-1&0&0&0\\
  0&-1&0&0&0&1\\
  0&0&0&0&0&0\\
  0&0&0&0&0&0\\
  0&0&1&0&0&0
 \end{pmatrix}:
 (\KMW3)^{\oplus6}\longrightarrow(\KMW2)^{\oplus5}.
\end{equation}
In particular,
\[
 \operatorname{im}\partial_3^{\Aone}(SL_4/B)
 =
 (\eta\KMW3)(e_{32}-e_{12})
 \oplus(\eta\KMW3)e_{13}.
\]
\end{proposition}

\begin{proof}
The exchange condition shows that deleting the middle letter in \(121\)
or \(232\) is nonreduced and that every other deletion in the six source
words is a cover.  Applying
\eqref{eq:tail-data}--\eqref{eq:sigma-data} gives the displayed
cocharacters and weights.  Concretely, one repeatedly uses
\(s_i(\alpha_j^\vee)=
\alpha_j^\vee-\langle\alpha_i,\alpha_j^\vee\rangle\alpha_i^\vee\);
since type \(A_3\) is simply laced, \Cref{lem:sigma-height} then checks
every weight by \(\sigma=1-\operatorname{ht}(\lambda)\).
Every deleted word is already the chosen
target word except \(321\) with \(q=2\): it gives \(31\), while the
chosen target word is \(13\).  The single commutation has degree
\(\ang{-1}\), proving the last column.

Under augmentation, an even \(\sigma\) gives zero and an odd \(\sigma\)
gives multiplication by \(\eta\).  There are four odd-weight entries.
Including the deletion signs and using
\(\ang{-1}\eta=-\eta\) gives \eqref{eq:SL4-flag-d3}.  Its two nonzero
row directions give the asserted image sheaf.
\end{proof}

\begin{proposition}
\label{prop:A3-diamond}
The real signature of the frame factor in \Cref{ex:SL4-frame} is
detected by square zero in a genuine degree-three Bruhat diamond.  More
precisely, use the normal words
\[
 a=s_3s_2s_1,\qquad b_1=s_2s_1,\qquad
 b_2=s_1s_3,\qquad c=s_1.
\]
Here the deleted word \(31\) for \(b_2\) is rewritten to the chosen
normal word \(13\).  Put
\(c_I=\prod_{i\in I}\alpha_i^\vee(-1)\in M\).  In the realized
right-\(\Z[M]\) complex, the four relevant cover entries are
\[
\begin{array}{c@{\qquad}c}
 \delta_{a,b_1}=c_{123}-1,
   &\delta_{b_1,c}=-c_{12}-1,\\[1mm]
 \delta_{a,b_2}=-c_{12}-1,
   &\delta_{b_2,c}=1-c_3.
\end{array}
\]
Consequently the coefficient of \(c\) in
\(\delta_2\delta_3(a)\) vanishes:
\[
 (-c_{12}-1)(c_{123}-1)
 +(1-c_3)(-c_{12}-1)=0.
\]
If the commuting-word degree
\(\vartheta_{a,2}=\ang{-1}\) is omitted, the two middle signs do not
give this cancellation.
\end{proposition}

\begin{proof}
For \(a\to b_1\), the deleted root has transported coroot
\(\alpha_1^\vee+\alpha_2^\vee+\alpha_3^\vee\) and
\(\sigma=-2\).  For \(b_1\to c\), the data are
\(\lambda=\alpha_1^\vee+\alpha_2^\vee\) and \(\sigma=-1\).
For \(a\to b_2\), deletion at \(q=2\) gives the word \(31\), whereas
the chosen target word is \(13\); the ordinary deletion sign and the
frame degree \(\ang{-1}\) cancel, while
\(\lambda=\alpha_1^\vee+\alpha_2^\vee\) and \(\sigma=-1\).
Finally \(b_2\to c\) deletes the last letter and has
\(\lambda=\alpha_3^\vee\), \(\sigma=0\), and ordinary sign \(-1\).
\Cref{prop:D-real} gives the four displayed entries.  Since
\(c_{123}=c_{12}c_3\) and \(c_{12}^2=1\), their sum is
\[
 (-c_{12}-1)c_3(c_{12}-1)=0.
\]
\end{proof}

\begin{remark}
The entries of \eqref{eq:SL3-complex} are sheaf morphisms, not elements
of a PID\@.  Consequently, applying Smith normal form to the displayed
real group-ring matrix does not compute the motivic homology sheaves.
\end{remark}

\section{The complete complex for \texorpdfstring{\(G_2\)}{G2}}
\label{sec:G2}

This section carries out the exceptional rank-two case in full.
Type \(G_2\) simultaneously exhibits the longest root strings and the
six-term braid, and it furnishes a complete example beyond types
\(A\) and \(B_2\).  Let \(G\) be split simply connected of type
\(G_2\), with \(\alpha_1\) short and
\(-\langle\alpha_2,\alpha_1^\vee\rangle=3\),
\(-\langle\alpha_1,\alpha_2^\vee\rangle=1\).  The Weyl group is
dihedral of order twelve.  Every element other than \(w_0\) has a
unique reduced word, necessarily alternating; we choose these as
normal words, together with \(\mathbf w_0=(1,2,1,2,1,2)\).

Deleting an interior letter of an alternating word leaves two equal
adjacent letters, which is not reduced.  Hence every reduced deletion
removes the first or the last letter, the deleted word is again
alternating and equal to the chosen normal word of its target, and
consequently
\[
 \vartheta_{a,q}=\ang1
 \qquad\text{for every cover in this complex.}
\]
The six-term braid degree \(\ang{-1}\) of
\Cref{prop:rank-two-degrees} does not appear on any single cover here;
it is the degree of the comparison between the two \(w_0\)-charts and
enters \Cref{prop:choices} when the opposite normal word
\((2,1,2,1,2,1)\) is chosen instead.

\begin{proposition}\label{prop:G2-covers}
With the conventions above, the nonzero components of the boundary of
\(\CellC_*(G/U)\) are exactly the following, where each row lists
\((a,q,b,\lambda_{a,q},\sigma_{a,q})\) and the component equals
\((-1)^{q-1}\Dop_{\lambda,\sigma}^{(p)}\operatorname{cut}_{a,q}\):
\[
\begin{array}{c c c c r}
\toprule
a&q&b&\lambda&\sigma\\
\midrule
1&1&e&\alpha_1^\vee&0\\
2&1&e&\alpha_2^\vee&0\\
12&1&2&\alpha_1^\vee+3\alpha_2^\vee&-3\\
12&2&1&\alpha_2^\vee&0\\
21&1&1&\alpha_1^\vee+\alpha_2^\vee&-1\\
21&2&2&\alpha_1^\vee&0\\
121&1&21&2\alpha_1^\vee+3\alpha_2^\vee&-4\\
121&3&12&\alpha_1^\vee&0\\
212&1&12&\alpha_1^\vee+2\alpha_2^\vee&-2\\
212&3&21&\alpha_2^\vee&0\\
1212&1&212&2\alpha_1^\vee+3\alpha_2^\vee&-4\\
1212&4&121&\alpha_2^\vee&0\\
2121&1&121&\alpha_1^\vee+2\alpha_2^\vee&-2\\
2121&4&212&\alpha_1^\vee&0\\
12121&1&2121&\alpha_1^\vee+3\alpha_2^\vee&-3\\
12121&5&1212&\alpha_1^\vee&0\\
21212&1&1212&\alpha_1^\vee+\alpha_2^\vee&-1\\
21212&5&2121&\alpha_2^\vee&0\\
121212&1&21212&\alpha_1^\vee&0\\
121212&6&12121&\alpha_2^\vee&0\\
\bottomrule
\end{array}
\]
Every weight satisfies
\(\sigma=1-\operatorname{ht}(\lambda)\), in accordance with
\Cref{lem:sigma-height}.
\end{proposition}

\begin{proof}
Interior deletions are nonreduced as noted, so the listed rows are all
covers.  The cocharacters are computed from
\eqref{eq:tail-data} by iterating
\(s_i(\mu)=\mu-\langle\alpha_i,\mu\rangle\alpha_i^\vee\) with the
\(G_2\) Cartan pairings, and the weights from
\eqref{eq:sigma-data}; each row is a two-line calculation, and
\Cref{lem:sigma-height} confirms every weight.  For instance, for
\(a=121\), \(q=1\): the tail is \(t=s_2s_1\), so
\(\lambda=s_1s_2(\alpha_1^\vee)
=s_1(\alpha_1^\vee+3\alpha_2^\vee)
=-\alpha_1^\vee+3(\alpha_1^\vee+\alpha_2^\vee)
=2\alpha_1^\vee+3\alpha_2^\vee\), of height \(5\), giving
\(\sigma=-4\).
\end{proof}

\begin{proposition}
\label{prop:G2-realized}
Let \(k=\R\), so \(K\) is the maximal compact subgroup \(SO(4)\) of
split \(G_2(\R)\) and \(M\cong(\Z/2)^2\).  The realized extended-Weyl
complex determined by \Cref{prop:G2-covers,thm:real-comparison} has
\(48\) generators, satisfies \(\partial^2=0\), and its integral
homology is
\[
 H_*\cong
 \bigl(\Z,\ \Z/2,\ 0,\ \Z^2,\ \Z/2,\ 0,\ \Z\bigr)
 \qquad(*=0,\ldots,6),
\]
which agrees with
\(H_*(SO(4);\Z)\cong H_*\bigl(S^3\times\R\mathrm P^3;\Z\bigr)\).
After flag augmentation the same data give the integral homology
\(\bigl(\Z,\ (\Z/2)^2,\ 0,\ \Z^2,\ (\Z/2)^2,\ 0,\ \Z\bigr)\) of the
full real flag manifold of \(G_2\), consistent with the classical
computations of \cite{Kocherlakota1995,RabeloSanMartin2019}.
\end{proposition}

\begin{proof}
By \Cref{prop:D-real} and \Cref{thm:real-comparison}, the cover with
 data \((\lambda,\sigma)\) realizes to right multiplication by
 \((-1)^{q-1}\bigl((-1)^\sigma\lambda(-1)-1\bigr)\) on \(\Z[M]\), with
 all frame signatures \(+1\) here.  Since \(|W|=12\) and
 \(|M|=4\), the resulting complex has \(48\) generators, and every
 matrix entry is explicitly determined by the table in
 \Cref{prop:G2-covers}.  The identity \(\partial^2=0\) is not an
 additional computational assumption: it follows from the cellular
 exact-couple construction and \Cref{thm:all-degree}.

 By \Cref{thm:real-comparison} this realized complex is the cellular
 chain complex of \(K\).  For split simply connected \(G_2\), one has
 \(K\cong SO(4)\), and the standard identification
 \(SO(4)\approx S^3\times\R\mathrm P^3\), followed by the K\"unneth
 formula, gives the displayed integral homology.  By
 \Cref{thm:flag-augmentation}, augmentation gives the cellular complex
 of \(K/M\); its displayed homology is the classical real-flag
 calculation of \cite{Kocherlakota1995,RabeloSanMartin2019}.
\end{proof}

\begin{remark}\label{rem:rank-two-diamonds}
The chain condition is intrinsic to the cellular exact couple.  In
rank two, the alternating-word description and
\Cref{prop:rank-two-degrees} give all frame data, while
\Cref{thm:all-degree} makes every Bruhat-diamond cancellation a
specialization of the exact-couple identity.  The exceptional cover
data are displayed completely in \Cref{prop:G2-covers}; the first
genuinely three-dimensional sign configuration is checked separately
in \Cref{prop:A3-diamond}.  These results give a self-contained finite
audit of the sign conventions used in the displayed examples.
\end{remark}

\section{The torus-support filtration}
\label{sec:filtration}

Choose the simple-coroot decomposition \(T\cong\Gm^r\).  In
\eqref{eq:HT-decomp}, write
\[
 H_S=
 \begin{cases}
  \Z,&S=\varnothing,\\
  \KMW{|S|},&S\ne\varnothing,
 \end{cases}
\]
and define
\[
  F^q\HT=
  \bigoplus_{|S|\geq q}H_S.
\]
Multiplication takes supports to their union, so each \(F^q\HT\) is an
ideal and
\[
  \HT=F^0\HT\supset F^1\HT\supset\cdots\supset
  F^{r+1}\HT=0.
\]

\begin{proposition}
\label{prop:support-ss}
The filtration
\[
  F^q\CellC_p(X)=
  \bigoplus_{\ell(a)=p}\KMW p\otimes F^q\HT
\]
is a finite filtration by subcomplexes.  It gives a strongly convergent
spectral sequence
\[
  E^1_{q,p-q}=
  H_p\bigl(\gr_F^q\CellC_*(X)\bigr)
  \Longrightarrow\CellH_p(X).
\]

Every associated graded differential is explicit.  For
\(S\subseteq\{1,\ldots,r\}\), let
\[
 T_S=\prod_{i\in S}\Gm,\qquad
 \lambda_S=\operatorname{pr}_S\lambda:\Gm\longrightarrow T_S.
\]
The summand \(H_S\) is naturally a module over \(\ZA[T_S]\).  Define
\[
\begin{split}
 \Dop_{\lambda,\sigma;S}^{(p)}:
 \KMW1\otimes\KMW{p-1}\otimes H_S
 &\longrightarrow\KMW{p-1}\otimes H_S,\\
 (u)\otimes x\otimes h
 &\longmapsto
 \ang{u^\sigma}x\otimes[\lambda_S(u)]h-x\otimes h.
\end{split}
\]
Then
\[
 \gr_F^q\CellC_*(X)
 \cong\bigoplus_{|S|=q}\CellC_{*,S}
\]
as chain complexes, where \(\CellC_{p,S}\) has the summands
\(\KMW p\otimes H_S\) indexed by \(\ell(a)=p\), and its cover component
is
\[
 (-1)^{h-1}\vartheta_{a,h}\,
 \Dop_{\lambda_{a,h},\sigma_{a,h};S}^{(p)}
 \operatorname{cut}_{a,h}.
\]
In particular \(\CellC_{*,\varnothing}=\CellC_*(G/B)\), and
\[
 E^1_{q,p-q}\cong
 \bigoplus_{|S|=q}H_p\bigl(\CellC_{*,S}\bigr).
\]
\end{proposition}

\begin{proof}
The differential is right \(\HT\)-linear, so it preserves every ideal
\(F^q\HT\).  Finiteness gives strong convergence.  The quotient by
\(F^1\HT\) is torus augmentation, hence the flag complex by
\Cref{thm:flag-augmentation}.

It remains to identify the full associated graded, including repeated
support coordinates.  Write
\(\lambda=\sum_i n_i\alpha_i^\vee\).  On the exact-support summand
\(H_S\), expand
\[
 [\lambda(u)]=
 \prod_i\bigl(1+
 ([\alpha_i^\vee(u^{n_i})]-[1])\bigr).
\]
Every term involving an index outside \(S\) raises support and vanishes
in \(\gr_F^{|S|}\).  Terms involving indices in \(S\) act inside
\(H_S\); repeated use of an existing coordinate is precisely the
multiplication
\(\KMW2\xrightarrow{\eta}\KMW1\) from
\eqref{eq:HT-decomp}.  Thus the action induced on \(H_S\) is the full
\(T_S\)-action by \([\lambda_S(u)]\), not merely its linear part.
\Cref{thm:all-degree} now gives the displayed cover component.  No term
mixes two subsets of the same cardinality, proving the direct-sum
decomposition and the formula for \(E^1\).
\end{proof}

\begin{corollary}
\label{cor:real-support-graded}
Over \(\R\), put \(c_i=\alpha_i^\vee(-1)\),
\(x_i=c_i-1\), and \(x_S=\prod_{i\in S}x_i\).  The realization of
\(\gr_F^q\CellC_*(X)\) is the direct sum, over \(|S|=q\), of integral
Bruhat complexes with one generator \(x_Se_a\) for each \(a\in W\).
If a cover is obtained by deleting position \(h\), and
\[
 \lambda_{a,h}=\sum_i n_i\alpha_i^\vee,
\]
its scalar coefficient on the \(S\)-summand is
\begin{equation}\label{eq:real-support-coefficient}
 (-1)^{h-1}\sgn(\vartheta_{a,h})
 \left((-1)^{
 \sigma_{a,h}+\sum_{i\in S}n_i}-1\right).
\end{equation}
Hence every entry is \(0\) or \(\pm2\).  The \(S=\varnothing\) summand
is the ordinary real flag complex; the other summands are obtained by
twisting its parity rule by the coroot coordinates supported on \(S\).
\end{corollary}

\begin{proof}
Modulo terms of support larger than \(S\), a component
\(c_i^{n_i}\) acts on \(x_i\) by
\[
 c_i^{n_i}x_i=(-1)^{n_i}x_i
\]
and acts as the identity when \(i\notin S\).  Thus
\(\lambda_{a,h}(-1)\) acts on \(x_S\) by
\((-1)^{\sum_{i\in S}n_i}\).  Combining this with the determinant
factor \((-1)^{\sigma_{a,h}}\), the deletion sign, and the signature of
the frame degree gives \eqref{eq:real-support-coefficient}.
\end{proof}

\begin{example}
\label{ex:SL4-support-correction}
Take the \(SL_4\) cover \(123\to23\) obtained by deleting the first
letter.  By \Cref{prop:SL4-degree-three},
\[
 \lambda=\alpha_1^\vee+\alpha_2^\vee+\alpha_3^\vee,
 \qquad \sigma=-2,\qquad\vartheta=\ang1.
\]
Formula \eqref{eq:real-support-coefficient} gives zero on the
\(S=\varnothing\) summand, as does flag augmentation.  On an exact
support summand it gives
\[
 (-1)^{-2+|S|}-1
 =
 \begin{cases}
  -2,&|S|\text{ is odd},\\
  0,&|S|\text{ is even}.
 \end{cases}
\]
Thus this cover first appears in support one, with coefficient \(-2\)
on each of \(x_1e_{123},x_2e_{123},x_3e_{123}\).  It is a concrete
torus-labelled boundary that cannot be recovered from the flag complex.
\end{example}

Thus the support spectral sequence interpolates between the flag
boundary and the extended-Weyl boundary by finitely many explicitly
computable integer complexes, rather than only by a first-order torus
correction.

We conclude with a structural consequence of the all-degree formula:
after inverting \(2\), the realized filtration splits along the
characters of \(M\) and the support spectral sequence degenerates.
For a character \(\chi:M\to\{\pm1\}\), write
\(S_\chi=\{i:\chi(c_i)=-1\}\) and let \(C_*^{\chi}\) denote the free
\(\Z\)-complex with one generator \(e_a\) for each \(a\in W\) and
cover entries the integers
\begin{equation}\label{eq:chi-entries}
 \partial^{\chi}_{a,b}
 =\sum_{q}\,
 \epsPS(a,q)(-1)^{q}
 \bigl(1-(-1)^{\sigma_{a,q}}\chi(c_{a,q})\bigr)
 \in\{0,\pm2\},
\end{equation}
the sum being over reduced deletion positions producing \(b\).  For
the trivial character this is the integral boundary of the real flag
manifold \(K/M\); in general it is the boundary of \(K/M\) with local
coefficients in \(\Z_\chi\), the sign representation of \(M\)
composed with \(\chi\).

\begin{theorem}\label{thm:half-splitting}
Let \(k=\R\).  After inverting \(2\), the realized filtered complex of
\Cref{thm:real-comparison} splits canonically:
\begin{equation}\label{eq:half-splitting}
 C_*^{\mathrm{CW}}(K;\Z[\tfrac12])
 \;\cong\;
 \bigoplus_{\chi\in\widehat M}
 C_*^{\chi}\otimes\Z[\tfrac12],
\end{equation}
compatibly with the support filtration of \Cref{prop:support-ss}: the
\(\chi\)-summand lies in filtration \(|S_\chi|\) and projects
isomorphically onto the summand \(x_{S_\chi}\) of the associated
graded.  Consequently the realized support spectral sequence with
\(\Z[\tfrac12]\)-coefficients degenerates at \(E^1\), and
\[
 H_*\bigl(K;\Z[\tfrac12]\bigr)
 \cong
 \bigoplus_{\chi\in\widehat M}
 H_*\bigl(K/M;\Z_\chi[\tfrac12]\bigr).
\]
In particular, every nonzero differential of the integral realized
support spectral sequence is \(2\)-power torsion.
\end{theorem}

\begin{proof}
Since \(M\cong(\Z/2)^r\) is elementary abelian, the idempotents
\[
 e_\chi=\prod_{i=1}^r\frac{1+\chi(c_i)\,c_i}2
 \in\Z[\tfrac12][M]
\]
are orthogonal, sum to \(1\), and split
\(\Z[\tfrac12][M]\cong\prod_\chi\Z[\tfrac12]\).  Every differential of
the realized complex is right multiplication by an element of
\(\Z[M]\) (\Cref{thm:real-comparison}), hence commutes with right
multiplication by \(e_\chi\); this gives the direct-sum decomposition
\eqref{eq:half-splitting}, with the entry of the \(\chi\)-summand
obtained by applying \(\chi\) to \eqref{eq:PS-group-ring}, which is
\eqref{eq:chi-entries}.  Writing \(c_i=1+x_i\) gives
\[
 e_\chi=2^{-r}
 \prod_{i\notin S_\chi}(2+x_i)
 \prod_{i\in S_\chi}(-x_i)
 =(-2)^{-|S_\chi|}\Bigl(x_{S_\chi}
 +\text{terms of support}>|S_\chi|\Bigr),
\]
so \(e_\chi\in F^{|S_\chi|}\otimes\Z[\tfrac12]\) and its class in the
associated graded is the unit multiple
\((-2)^{-|S_\chi|}x_{S_\chi}\).  Since the \(2^r\) idempotents have
pairwise distinct leading supports exhausting all subsets, the
splitting refines the filtration as claimed, and the induced map from
\eqref{eq:half-splitting} to the associated graded of
\(F^\bullet\otimes\Z[\tfrac12]\) is an isomorphism of complexes.  A
finitely filtered complex isomorphic to its associated graded has
degenerate spectral sequence, giving the displayed homology
decomposition; the identification of the \(\chi\)-summand homology
with the \(\Z_\chi\)-local-coefficient homology of \(K/M\) is the
standard character decomposition of the chain complex of the free
\(M\)-cover \(K\to K/M\) after inverting \(|M|\).  The final claim
follows because the integral spectral sequence becomes degenerate
after inverting \(2\).
\end{proof}

\begin{remark}
The trivial-character summand recovers
\(H_*(K/M;\Z[\tfrac12])\), and the full decomposition is a cellular
refinement, uniform in all split types, of the transfer isomorphism
for the finite cover \(K\to K/M\).  Its content here is not the
transfer itself but its compatibility with the motivic support
filtration: the theorem identifies the character splitting with the
associated graded of a filtration that is defined before realization,
on the level of \(\HT\)-coefficients.  Away from the prime \(2\), the
torus-enriched boundary is therefore completely computed by the
parity-twisted flag matrices \eqref{eq:chi-entries}; all integral
subtlety detected by this filtration is \(2\)-primary.  In the
even-flag and type-\(A\) ranges this is compatible with the sharper
order-\(2\) torsion results of \cite{Matszangosz2021,
HudsonMatszangoszWendt2024}; no all-type assertion excluding higher
\(2\)-power torsion is used here.
\end{remark}

\appendix

\section{Dictionary between the three boundary conventions}
\label{app:dictionary}

For convenience, we collect the indexing and coefficient dictionary.
Let \(a=s_{i_1}\cdots s_{i_p}\), delete position \(q\), and put
\[
  \lambda=t_q^{-1}\alpha_{i_q}^\vee,\qquad
  \sigma=\sum_{\beta\in\Phi^+\cap t_q\Phi^-}
  \langle\beta,\alpha_{i_q}^\vee\rangle.
\]
Then:
\[
\begin{array}{p{0.25\textwidth}p{0.67\textwidth}}
\toprule
setting & cover component, up to the common orientation factor\\
\midrule
basic affine / group &
\((u)\otimes x\otimes[t]\mapsto
 \ang{u^\sigma}x\otimes[\lambda(u)t]-x\otimes[t]\)\\[1mm]
maximal compact &
\((-1)^\sigma[\lambda(-1)]-[1]\in\Z[M]\)\\[1mm]
flag quotient &
\((u)\otimes x\mapsto(\ang{u^\sigma}-1)x\), a power-map
multiple of \(\eta\)\\
\bottomrule
\end{array}
\]
The first row retains both endpoints.  The second records their real
component labels.  The third identifies those labels and therefore
combines the endpoints.

\bibliographystyle{alpha}
\bibliography{references}

\end{document}